\documentclass[reqno]{amsart}

\usepackage{todonotes}

\usepackage{graphicx,float}
\usepackage{amsthm,amsmath,amssymb,amstext,amsfonts,color}

\usepackage{fullpage}

\usepackage{etoolbox}
\patchcmd{\thebibliography}{\section*{\refname}}{}{}{}

\usepackage[all]{xy} 

\newtheorem{Theorem}{Theorem}[section]

\newtheorem{Proposition}[Theorem]{Proposition}
\newtheorem{Corollary}[Theorem]{Corollary}
\newtheorem{Lemma}[Theorem]{Lemma}
\newtheorem{Definition}[Theorem]{Definition}
\theoremstyle{remark}
\newtheorem{Remark}[Theorem]{Remark}
\newtheorem{ex}{Example}

\definecolor{foge}{rgb}{0.1, 0.6, 0.1}

\numberwithin{equation}{section}

\allowdisplaybreaks

\title{Combinatorial approach to Andrews--Gordon and Bressoud type identities}

\author{Jehanne Dousse}
\address{Université de Genève, rue du Conseil Général 7--9, 1205 Genève, Switzerland}
\email{jehanne.dousse@unige.ch}

\author{Fr\'ed\'eric Jouhet}
\address{Univ Lyon, Université Claude Bernard Lyon 1, UMR5208, Institut Camille Jordan, F-69622 Villeurbanne, France}
\email{jouhet@math.univ-lyon1.fr}

\author{Isaac Konan}
\address{Univ Lyon, Université Claude Bernard Lyon 1, UMR5208, Institut Camille Jordan, F-69622 Villeurbanne, France}
\email{konan@math.univ-lyon1.fr}



\begin{document}

\begin{abstract}
We provide combinatorial tools inspired by work of Warnaar to give combinatorial interpretations of the sum sides of the Andrews--Gordon and Bressoud identities. More precisely, we give an explicit weight- and length-preserving bijection between sets related to integer partitions, which provides these interpretations. In passing, we discover the $q$-series version of an identity of Kur\c{s}ung\"oz, similar to the Bressoud identity but with opposite parity conditions, which we prove combinatorially using the classical Bressoud identity and our bijection. 
We also use this bijection to prove combinatorially many identities, some known and other new, of the Andrews--Gordon and Bressoud type.
\end{abstract}

\maketitle

\section{Introduction}

For a non-negative integer $n$, an integer partition of $n$ is a finite non-increasing sequence of positive integers $\lambda=(\lambda_1,\ldots,\lambda_\ell)$ whose sum is $n$; the integers $\lambda_i$ are called the parts of $\lambda$ and $\ell$ is its length.

The Rogers--Ramanujan identities \cite{RR19}, stated here in the combinatorial version due to MacMahon \cite{MacMahon} and Schur \cite{SchurRR}, are the following.
\begin{Theorem}[Rogers--Ramanujan identities, partition version]
\label{th:RRcomb}
Let $i=1$ or $2$. For all non-negative integers $n$, the number of partitions of $n$ such that the difference between two consecutive parts is at least $2$ and the part $1$ appears at most $i-1$ times is equal to the number of partitions of $n$ into parts congruent to $\pm (2+i) \mod 5.$
\end{Theorem}

These identities are central in combinatorics and number theory, see the book \cite{Sills} and references therein. Moreover they appear naturally in many other fields: the representation theory of affine Lie algebras \cite{LM,Lepowsky,Lepowsky2}, statistical mechanics \cite{Baxter}, algebraic geometry and arc spaces \cite{Mourtada}, knot theory \cite{Knot}, and others.

\medskip

In 1961, Gordon~\cite{Go} proved the following combinatorial result, which extends both Rogers--Ramanujan partition identities.

\begin{Theorem}[Gordon's identities]\label{th:Gordon}
Let $r$ and $i$ be integers such that $r\geq 2$ and $1\leq i \leq r$. Let $\mathcal{T}_{i,r}$ be the set of partitions $\lambda=(\lambda_1,\lambda_2,\dots, \lambda_\ell)$ where $\lambda_{j}-\lambda_{j+r-1} \geq 2$ for all $j$, and at most $i-1$ of the parts $\lambda_j$ are equal to $1$. Let $\mathcal{E}_{i,r}$ be the set of partitions whose parts are not congruent to $0,\pm i \mod (2r+1)$. Let $n$ be a non-negative integer, and let $T_{i,r}(n)$ (respectively $E_{i,r}(n)$) denote the number of partitions of $n$ which belong to $\mathcal{T}_{i,r}$ (respectively $\mathcal{E}_{i,r}$). Then we have 
$$T_{i,r}(n)=E_{i,r}(n).$$
\end{Theorem}

The Rogers--Ramanujan identities correspond to the cases $r=i=2$ and $r=i+1=2$ in Theorem~\ref{th:Gordon}. Recall some standard notations for $q$-series which can be found in~\cite{GR}. Let $q$ be a fixed complex parameter with $|q|<1$.
The $q$-shifted factorial is defined for any complex
parameter $a$ by
\begin{equation*}
(a)_\infty\equiv (a;q)_\infty:=\prod_{j\geq 0}(1-aq^j)\;\;\;\;\mbox{and}\;\;\;\;(a)_k\equiv (a;q)_k:=\frac{(a;q)_\infty}{(aq^k;q)_\infty},
\end{equation*}
where $k$ is any integer.
Since the base $q$ is often the same throughout this paper,
it may be readily omitted (in notation, writing $(a)_k$ instead of $(a;q)_k$, etc.) which will not lead to any confusion. For brevity, write
\begin{equation*}
(a_1,\ldots,a_m;q)_k:=(a_1)_k\cdots(a_m)_k,
\end{equation*}
where $k$ is an integer or infinity. In~\cite{A74}, Andrews expressed Gordon's identities as $q$-series identities.

\begin{Theorem}[Andrews--Gordon identities]\label{th:AGseries}
Let $r \geq 2$ and $1 \leq i \leq r$ be two integers. We have
\begin{equation}\label{eq:AGri}
\sum_{s_1\geq\dots\geq s_{r-1}\geq0}\frac{q^{s_1^2+\dots+s_{r-1}^2+s_{i}+\dots+s_{r-1}}}{(q)_{s_1-s_2}\dots(q)_{s_{r-2}-s_{r-1}}(q)_{s_{r-1}}}=\frac{(q^{2r+1},q^{i},q^{2r-i+1};q^{2r+1})_\infty}{(q)_\infty}.
\end{equation}
\end{Theorem}

Just like the Rogers--Ramanujan identities, the Andrews--Gordon identities also arise in several fields, such as representation theory \cite{DHK,MP1,MP2,W2} or commutative algebra \cite{Pooneh,ADJM}, to name only a few.

One immediately sees that the generating function of the set $\mathcal{E}_{i,r}$ in Theorem~\ref{th:Gordon} is given by the product on the right-hand side of~\eqref{eq:AGri}. However, showing that the left-hand side is the generating function of $\mathcal{T}_{i,r}$ is not that simple. Originally, it was proved by Andrews \cite{A74} using recurrences. The first and only bijective proof was given by Warnaar~\cite{W} in a more general context.

\medskip

In~\cite{B}, Bressoud proved the following result, which is considered to be the even moduli counterpart of Gordon's identities.

\begin{Theorem}[Bressoud's identities]\label{th:Bressoud}
Let $r$ and $i$ be integers such that $r\geq 2$ and $1\leq i < r$. Let $\mathcal{U}_{i,r}$ be the set of partitions $\lambda=(\lambda_1,\lambda_2,\dots, \lambda_\ell)$ where $\lambda_{j}-\lambda_{j+r-1} \geq 2$ for all $j$, $\lambda_{j}-\lambda_{j+r-2} \leq 1$ only if $\lambda_{j}+\lambda_{j+1}+\cdots+\lambda_{j+r-2} \equiv i-1 \mod 2$, and at most $i-1$ of the parts $\lambda_j$ are equal to $1$. Let $\mathcal{F}_{i,r}$ be the set of partitions whose parts are not congruent to $0,\pm i \mod (2r)$. Let $n$ be a non-negative integer, and let $U_{i,r}(n)$ (respectively $F_{i,r}(n)$) denote the number of partitions of $n$ which belong to $\mathcal{U}_{i,r}$ (respectively $\mathcal{F}_{i,r}$). Then we have 
$$U_{i,r}(n)=F_{i,r}(n).$$
\end{Theorem}
The $q$-series counterpart of Theorem~\ref{th:Bressoud}, also proved in \cite{B}, which is true for $1\leq i \leq r$, is 
\begin{equation}\label{eq:Bri}
\sum_{s_1\geq\dots\geq s_{r-1}\geq0}\frac{q^{s_1^2+\dots+s_{r-1}^2+s_{i}+\dots+s_{r-1}}}{(q)_{s_1-s_2}\dots(q)_{s_{r-2}-s_{r-1}}(q^2;q^2)_{s_{r-1}}}=\frac{(q^{2r},q^{i},q^{2r-i};q^{2r})_\infty}{(q)_\infty}.
\end{equation}
Again, the right-hand side of~\eqref{eq:Bri} is clearly the generating function of the set $\mathcal{F}_{i,r}$.
We extend the definition of $F_{i,r}(n)$ to $i=r$ by setting $F_{r,r}(n)$ to be the coefficient of $q^n$ in the right-hand side of \eqref{eq:Bri}. On the other hand, $U_{r,r}(n)$ is the number of partitions in $\mathcal{U}_{r,r}$, where $\mathcal{U}_{r,r}$ is defined as in Theorem~\ref{th:Bressoud}.

Similarly to the Andrews--Gordon case, one can wonder whether there is a bijective proof that the left-hand side of~\eqref{eq:Bri} is the generating function of the set $\mathcal{U}_{i,r}$. As for the Andrews--Gordon identities, it was originally proved via recurrences \cite{B}.
One of the goals of this paper is to provide such a bijective proof.
To do so, we use Warnaar's point of view in~\cite{W}, which describes partitions by their multiplicity sequences. Actually, our main result is a general bijection between two sets related to partitions, from which we derive many corollaries, among which the desired bijective proof.

\medskip

To prove our main result, we extend the definition of integer partitions to allow parts equal to $0$.
\textit{Thus, in the remainder of the paper, a partition denotes a finite non-increasing sequence of non-negative integers}. For such partitions, we consider the multiplicity (or frequency) sequence $(f_u)_{u\geq 0}$, where $f_u$ is the number of occurrences of the part $u$ in the partition. Then a partition $\lambda$ can be described equivalently as the finite sequence of non-negative integers $(\lambda_1,\ldots,\lambda_\ell)$ of its parts, or as the infinite sequence of non-negative integers $(f_u)_{u\geq 0}$ of its multiplicities (where there are finitely many positive terms).
For examples, in terms of frequencies, the partition $(4,4,3,1,0)$ would be written as $(1,1,0,1,2,0,\dots)$.
Moreover, we associate with a partition $\lambda$ the classical weight statistic
$$|\lambda|=\lambda_1+\lambda_2+\cdots+\lambda_\ell = \sum_{u\geq 0} uf_u.$$

For an integer $r\geq 2$, we define the following set of partitions:
\begin{equation}\label{defAr}
\mathcal{A}_{r}:=\{(f_u)_{u\geq 0}\,|\,f_0\leq r-1\; \mbox{and}\; f_u+f_{u+1}\leq r-1\;\mbox{for all}\; u\}.
\end{equation} 

Let $s_1\geq \cdots \geq s_{r-1}\geq 0$ be integers, and set $s_0=\infty$ and $s_r= 0$. 
\begin{Definition}\label{def:mu}
Denote by $\mu(s_1,\ldots,s_{r-1})=(f_u)_{u\geq 0}$ the partition such that for all $j\in \{0,\ldots,r-1\}$, 
\begin{equation}\label{eq:defmu}
(f_{2u},f_{2u+1}) = (j,0)\text{ for all } s_{j+1}\leq u < s_j.
\end{equation}
\end{Definition}

Note that $\mu(s_1,\ldots,s_{r-1})\in\mathcal{A}_{r}$, and that its multiplicity sequence $(f_u)_{u\geq 0}$ has the form 
$$(\underbrace{r-1,0, \ldots, r-1, 0}_{s_{r-1}\text{ pairs}},\ldots,\underbrace{j,0, \ldots, j,0}_{s_{j}-s_{j+1}\text{ pairs}},\ldots,\underbrace{1,0, \ldots, 1,0}_{s_{1}-s_{2}\text{ pairs}}, 0 , \ldots ).$$

 \begin{Definition}\label{def:P} 
Denote by $\mathcal{P}(s_1,\ldots,s_{r-1})$ the set of sequences $\lambda=(\lambda_0,\ldots,\lambda_{s_1-1})$ of non-negative integers such that for all $j\in \{1,\ldots,r-1\}$, the sequence $\lambda^{(j)}:=(\lambda_{s_j-1},\ldots,\lambda_{s_{j+1}})$ is a partition.
\end{Definition}

Finally let 
\begin{equation*}\label{defPr}
\mathcal{P}_r:= \bigsqcup_{s_1\geq \cdots \geq s_{r-1}\geq 0} \{\mu(s_1,\ldots,s_{r-1})\}\times \mathcal{P}(s_1,\ldots,s_{r-1}).
\end{equation*}
The weight of an element $(\mu(s_1,\ldots,s_{r-1}),\lambda)$ of $\mathcal{P}_r$ is defined to be $|\mu(s_1,\ldots,s_{r-1})|+|\lambda^{(1)}|+\dots+|\lambda^{(r-1)}|$. Its length is defined to be the length of $\mu(s_1,\ldots,s_{r-1})$, i.e. $s_1+\dots+s_{r-1}$.

Now we are ready to state the main result of this paper.

\begin{Theorem}[Bijection]\label{th:bij}
For all $r\geq 2$, there is an explicit weight- and length-preserving bijection between the sets $\mathcal{P}_r$ and $\mathcal{A}_r$.
\end{Theorem}

The precise description of this bijection is provided in Section~\ref{sec:Bijection}. The first consequence of this result is a simplification of Warnaar's proof \cite{W} of the connection between Theorem~\ref{th:Gordon} and~\eqref{eq:AGri}. It also yields bijectively that the left-hand side of~\eqref{eq:Bri} is indeed the generating function of the set $\mathcal{U}_{i,r}$ from Bressoud's Theorem~\ref{th:Bressoud}.

\begin{Corollary}[Sum sides of the Andrews--Gordon and Bressoud identities]\label{coro:GF}
For $r$ and $i$ integers such that $r\geq 2$ and $1\leq i \leq r$, we have  the following generating functions:
\begin{equation}\label{eq:AGW}
\sum_{n\geq 0} T_{i,r}(n)q^n=\sum_{s_1\geq\dots\geq s_{r-1}\geq0}\frac{q^{s_1^2+\dots+s_{r-1}^2+s_{i}+\dots+s_{r-1}}}{(q)_{s_1-s_2}\dots(q)_{s_{r-2}-s_{r-1}}(q)_{s_{r-1}}},
\end{equation}
and
\begin{equation}\label{eq:BriW}
\sum_{n\geq 0} U_{i,r}(n)q^n=\sum_{s_1\geq\dots\geq s_{r-1}\geq0}\frac{q^{s_1^2+\dots+s_{r-1}^2+s_{i}+\dots+s_{r-1}}}{(q)_{s_1-s_2}\dots(q)_{s_{r-2}-s_{r-1}}(q^2;q^2)_{s_{r-1}}}.
\end{equation}
\end{Corollary}

It is natural to look for an identity similar to Bressoud's but with opposite parity conditions. This was done by Kur\c{s}ung\"oz in \cite{Kur} and then arised again as so-called ``ghost series" in \cite{KLRS}. However, while Kur\c{s}ung\"oz had an expression for the generating function as a sum of products (\eqref{eq:fij} below), our expression  as a multisum is new.

\begin{Corollary}[Kur\c{s}ung\"oz identities, new multisum]\label{coro:newbressoud}
Let $r$ and $i$ be integers such that $r\geq 2$ and $1\leq i \leq r$. Let $\tilde{\mathcal{U}}_{i,r}$ be the set of partitions $\lambda=(\lambda_1,\lambda_2,\dots, \lambda_\ell)$ where $\lambda_{j}-\lambda_{j+r-1} \geq 2$ for all $j$, $\lambda_{j}-\lambda_{j+r-2} \leq 1$ only if $\lambda_{j}+\lambda_{j+1}+\cdots+\lambda_{j+r-2} \equiv i \mod 2$, and at most $i-1$ of the parts $\lambda_j$ are equal to $1$. For any non-negative integer $n$, let $\tilde{U}_{i,r}(n)$ denote the number of partitions of $n$ which belong to $\tilde{\mathcal{U}}_{i,r}$. Then, by setting $F_{r+1,r}(n)=F_{r-1,r}(n)$ and $F_{0,r}(n)=0$, we have 
$$\tilde{U}_{i,r}(n)+\tilde{U}_{i,r}(n-1)=F_{i+1,r}(n)+F_{i-1,r}(n-1).$$
Moreover
\begin{multline}\label{eq:fij}
(1+q)\sum_{s_1\geq\dots\geq s_{r-1}\geq0}\frac{q^{s_1^2+\dots+s_{r-1}^2+s_{i}+\cdots+ s_{r-2}+2s_{r-1}}}{(q)_{s_1-s_2}\dots(q)_{s_{r-2}-s_{r-1}}(q^2;q^2)_{s_{r-1}}}\\
=\frac{1}{(q)_\infty}\left((q^{2r},q^{i+1},q^{2r-i-1};q^{2r})_\infty+q(q^{2r},q^{i-1},q^{2r-i+1};q^{2r})_\infty\right),
\end{multline}
and 
\begin{equation}\label{eq:fijW}
\sum_{n\geq 0} \tilde{U}_{i,r}(n)q^n=\sum_{s_1\geq\dots\geq s_{r-1}\geq0}\frac{q^{s_1^2+\dots+s_{r-1}^2+s_{i}+\cdots+ s_{r-2}+2s_{r-1}}}{(q)_{s_1-s_2}\dots(q)_{s_{r-2}-s_{r-1}}(q^2;q^2)_{s_{r-1}}}.
\end{equation}
\end{Corollary}

Note that by Theorem~\ref{th:Bressoud} and Corollary~\ref{coro:newbressoud}, we have for all non-negative integers $n$ the equalities
$$\tilde{U}_{i,r}(n)+\tilde{U}_{i,r}(n-1)=U_{i+1,r}(n)+U_{i-1,r}(n-1).$$

Actually, by using~\eqref{eq:AGri} and~\eqref{eq:Bri} and studying the image of several subsets of $\mathcal{A}_r$ by our bijection in Theorem~\ref{th:bij}, we are able to derive combinatorially the following list of Andrews--Gordon and Bressoud type identities.

\begin{Corollary}[Andrews--Gordon and Bressoud type identities]\label{coro:list}
For any integer $r\geq 2$, we have 
\begin{eqnarray}
&&\hskip -1.5cm\sum_{s_1\geq \dots\geq s_{r-1}\geq 0}\frac{q^{s_1^2+\dots+s_{r-1}^2 -s_{1}-\cdots-s_{i}}}{(q)_{s_1-s_2}\dots(q)_{s_{r-2}-s_{r-1}}(q)_{s_{r-1}}}=\sum_{k=0}^{i}\frac{(q^{2r+1},q^{r-i+k},q^{r+i-k+1};q^{2r+1})_\infty}{(q)_\infty},\label{Br3.3form}\\
&&\hskip -1.5cm\sum_{s_1\geq \dots\geq s_{r-1}\geq 0}\frac{q^{s_1^2+\dots+s_{r-1}^2 -s_{1}-\cdots-s_{i}}}{(q)_{s_1-s_2}\dots(q)_{s_{r-2}-s_{r-1}}(q^2;q^2)_{s_{r-1}}}=\sum_{k=0}^{i}\frac{(q^{2r},q^{r-i+2k},q^{r+i-2k};q^{2r})_\infty}{(q)_\infty}\label{Br3.5form},\\
&&\hskip -1.5cm\sum_{s_1\geq \dots\geq s_{r-1}\geq 0}\frac{q^{s_1^2+\dots+s_{r-1}^2 -s_{1}-\cdots-s_{i}+s_{r-1}}}{(q)_{s_1-s_2}\dots(q)_{s_{r-2}-s_{r-1}}(q^2;q^2)_{s_{r-1}}} =\sum_{k=0}^{i}\frac{(q^{2r},q^{r-i+2k-1},q^{r+i-2k+1};q^{2r})_\infty}{(q)_\infty}\label{fijform},\\
&&\hskip -1.5cm\sum_{s_1\geq \dots\geq s_{r-1}\geq 0}\frac{q^{s_1^2+\dots+s_{r-1}^2 -s_{1}-\cdots-s_{i}}(1-q^{s_{i}})}{(q)_{s_1-s_2}\dots(q)_{s_{r-2}-s_{r-1}}(q)_{s_{r-1}}}= \frac{(q^{2r+1},q^{r-i},q^{r+i+1};q^{2r+1})_\infty}{(q)_\infty}\label{AGPform},\\
&&\hskip -1.5cm\sum_{s_1\geq \dots\geq s_{r-1}\geq 0}\frac{q^{s_1^2+\dots+s_{r-1}^2 -s_{1}-\cdots-s_{i}}(1-q^{s_{i}+s_{i-1}})}{(q)_{s_1-s_2}\dots(q)_{s_{r-2}-s_{r-1}}(q^2;q^2)_{s_{r-1}}} =2\frac{(q^{2r},q^{r-i},q^{r+i};q^{2r})_\infty}{(q)_\infty}\label{BPform},\\
&&\hskip -1.5cm\sum_{s_1\geq \dots\geq s_{r-1}\geq 0}\frac{q^{s_1^2+\dots+s_{r-1}^2 -s_{1}-\cdots-s_{i}}(1-q^{s_{i}+s_{r-1}})}{(q)_{s_1-s_2}\dots(q)_{s_{r-2}-s_{r-1}}(q^2;q^2)_{s_{r-1}}}= \frac{(q^{2r},q^{r-i},q^{r+i};q^{2r})_\infty}{(q)_\infty}\label{isaacform},\\
&&\hskip -1.5cm\sum_{s_1\geq \dots\geq s_{r-1}\geq 0}\frac{q^{s_1^2+\dots+s_{r-1}^2 -s_{1}-\cdots-s_{i}}(q^{s_{r-1}}-q^{s_i})}{(q)_{s_1-s_2}\dots(q)_{s_{r-2}-s_{r-1}}(q^2;q^2)_{s_{r-1}}}=\frac{(q^{2r},q^{r-i-1},q^{r+i+1};q^{2r})_\infty}{(q)_\infty}\label{fijform2},\\
&&\hskip -1.5cm\sum_{s_1\geq \dots\geq s_{r-1}\geq 0}\frac{q^{s_1^2+\dots+s_{r-1}^2 -s_{1}-\cdots-s_{i}+s_{r-1}}(1-q^{s_{i}+s_{i-1}})}{(q)_{s_1-s_2}\dots(q)_{s_{r-2}-s_{r-1}}(q^2;q^2)_{s_{r-1}}}\nonumber\\
&&\hskip 4cm =\frac{(q^{2r},q^{r-i-1},q^{r+i+1};q^{2r})_\infty+(q^{2r},q^{r-i+1},q^{r+i-1};q^{2r})_\infty}{(q)_\infty}\label{fijform3},
\end{eqnarray}
where $0\leq i\leq r-1$ for~\eqref{Br3.3form}--\eqref{fijform}, $1\leq i\leq r-1$ for~\eqref{AGPform},~\eqref{isaacform},~\eqref{fijform2}, and $2\leq i\leq r-1$ for~\eqref{BPform},~\eqref{fijform3}.
\end{Corollary} 

Identities~\eqref{Br3.3form} and~\eqref{Br3.5form}, together with~\eqref{eq:AGri} and~\eqref{eq:Bri}, were proven by Bressoud in~\cite{Br80} as special cases of a very general formula. In~\cite{DJK}, we proved and generalized all formulas~\eqref{eq:AGri}, \eqref{eq:Bri}, \eqref{eq:fij}, and~\eqref{Br3.3form}--\eqref{fijform} by using the Bailey lemma and lattice, and we explained why~\eqref{eq:fij} and~\eqref{fijform} are not consequences of Bressoud's general formula from~\cite{Br80}. In~\cite{ADJM}, a combinatorial conjecture of Afsharijoo arising from commutative algebra related to arc spaces was solved by using formula~\eqref{AGPform}, which is a direct consequence of~\eqref{Br3.3form}. It is also explained in~\cite{ADJM} how one can derive~\eqref{BPform} from~\eqref{Br3.5form}. One could also deduce similarly~\eqref{isaacform}--\eqref{fijform3} from~\eqref{Br3.5form} and~\eqref{fijform}. 

What we want to point out here is that our present approach yields all formulas in Corollary~\ref{coro:list} in a purely combinatorial way: indeed we prove that for all these formulas, both sides are generating functions of explicit subsets of $\mathcal{A}_r$ and $\mathcal{P}_r$, and our bijection from Theorem~\ref{th:bij} then yields the identities.

Recall also that the open problem of giving a combinatorial interpretation for the left-hand side of the aforementioned Bressoud general formula in~\cite{Br80} (when parameters have specific forms), known as Bressoud's conjecture, has been settled only recently by Kim~\cite{Ki} and He, Ji, and Zhao~\cite{HJZ1,HJZ2}. The main combinatorial tool they use is the so-called Gordon marking for partitions. Our method is different, as we do not use the Gordon marking. Moreover, although we do not prove a result as general as the former Bressoud conjecture, we manage to give combinatorial proofs of~\eqref{eq:fij} and~\eqref{fijform} which, as we already explained, are not special cases of Bressoud's result.

\medskip

This paper is organized as follows. In Section~\ref{sec:setup}, we give the combinatorial setup for our results by defining several sets of partitions and computing their generating functions. In Section~\ref{sec:Bijection}, we prove Theorem~\ref{th:bij} by giving the explicit bijection. Finally, in Section~\ref{sec:cor} we prove the three corollaries.

\section{The setup for our combinatorial approach}\label{sec:setup}

In this section, we define two types of combinatorial objects related to partitions, provided with a weight statistic. As will be seen, using either Gordon's Theorem~\ref{th:Gordon} or Bressoud's Theorem~\ref{th:Bressoud}, their generating functions are respectively the right and left-hand sides for the identities we are interested in, namely~\eqref{eq:AGri},~\eqref{eq:Bri},~\eqref{eq:fij}, and~\eqref{Br3.3form}--\eqref{fijform3}.

\subsection{Combinatorial description of the right-hand (product or sum of products) sides}

We first need some general results making the connection between the two combinatorial descriptions of partitions (in terms of parts and in terms of multiplicity sequences). Here we use the notations given in the introduction. 

In the literature, the set $\mathcal{T}_{i,r}$ of Theorem \ref{th:Gordon} is often described as the set of partitions $(f_u)_{u\geq 1}$ such that 
\begin{equation*}
\begin{cases}
f_1\leq i-1,\\
\text{for all }u\geq 1, \ f_u+f_{u+1}\leq r-1.
\end{cases}
\end{equation*}
This formulation is in particular more convenient when dealing with representation theory \cite{MP1} or Gr\"obner bases \cite{Pooneh}, and it will also be more suited to our combinatorial approach.

The following proposition states this type of correspondence between difference conditions on parts and restrictions on frequencies more generally, including the cases of $\mathcal{U}_{i,r}$ and $\tilde{\mathcal{U}}_{i,r}$.

\begin{Proposition}\label{prop:cdf}
Let $d,m$ be positive integers. Let $\lambda=(\lambda_1,\ldots,\lambda_\ell)=(f_u)_{u\geq 0}$ be a partition.
\begin{enumerate}
\item For all $1\leq k\leq \ell-m$,
\begin{equation*}\label{eq:conddif1}
\lambda_k-\lambda_{k+m}\geq d
\end{equation*}
if and only if for all $u\geq 0$,
\begin{equation*}\label{eq:conddif2}
f_u+f_{u+1}+\cdots+f_{u+d-1}\leq m.
\end{equation*}
\item Let $(P)$ be a property on integers. Then the following statements are equivalent:
\begin{equation*}
\label{eq:conddif1bis}
\begin{cases}
\text{for all } 1\leq k\leq \ell-m,\ \lambda_k-\lambda_{k+m}\geq d,\\
\text{for all } 1\leq k \leq \ell-m+1,\ \lambda_k-\lambda_{k+m-1}\leq d-1 \Rightarrow \lambda_k+\cdots+\lambda_{k+m-1}\text{ satisfies }(P),
\end{cases}
\end{equation*}
and
\begin{equation*}
\label{eq:conddif2bis}
\begin{cases}
\text{for all } u\geq 0,\ f_u+f_{u+1}+\cdots+f_{u+d-1}\leq m,\\
\text{for all } u\geq 0, \ f_u+\cdots+f_{u+d-1}= m \Rightarrow uf_u+\cdots+(u+d-1) f_{u+d-1}\text{ satisfies }(P).
\end{cases}
\end{equation*}
\end{enumerate}
\end{Proposition}

\begin{proof}
The first part is classical, and is simply a way to describe either in terms of frequencies or in terms of differences between parts the following fact: ``in each interval of integers of length $d$, there are at most $m$ parts of the partition''.

The second part follows from a similar reasoning. The first line of each statement is the same as in (1), so they are equivalent. Then ``$\text{for all } u\geq 0,f_u+\cdots+f_{u+d-1}= m$'' implies that ``$\text{for all } 1\leq k \leq \ell-m+1,\ \lambda_k-\lambda_{k+m-1}\leq d-1$''. And together with ``$\text{for all } u\geq 0,\ f_u+f_{u+1}+\cdots+f_{u+d-1}\leq m$'', the statement ``$\text{for all } 1\leq k \leq \ell-m+1,\ \lambda_k-\lambda_{k+m-1}\leq d-1$'' implies ``$\text{for all } u\geq 0,f_u+\cdots+f_{u+d-1}= m$''.
Finally, ``$\text{for all } 1\leq k \leq \ell-m+1, \lambda_k+\cdots+\lambda_{k+m-1}\text{ satisfies }(P)$'' and ``$\text{for all } u\geq 0, uf_u+\cdots+(u+d-1) f_{u+d-1}\text{ satisfies }(P)$'' are just two different ways to say that the sum of any $m$ consecutive parts of the partition satisfies $(P)$.
\end{proof}

Using Proposition \ref{prop:cdf}, one can describe the sets $\mathcal{T}_{i,r}$, $\mathcal{U}_{i,r}$ and $\tilde{\mathcal{U}}_{i,r}$ from the introduction in terms of frequencies. These formulations both for $\mathcal{T}_{i,r}$ and $\mathcal{U}_{i,r}$ are widely used in the literature, see e.g. \cite{B80}.

\begin{Proposition}\label{prop:TUfreq}
The set $\mathcal{T}_{i,r}$ described in Theorem~\ref{th:Gordon} consists of partitions $(f_u)_{u\geq 0}$ such that 
\begin{equation}\label{eq:cdt}
\begin{cases}
f_0 =0,\\
f_1\leq i-1,\\
\text{for all }u\geq 1, \ f_u+f_{u+1}\leq r-1.
\end{cases}
\end{equation}
The set $\mathcal{U}_{i,r}$ described in Theorem~\ref{th:Bressoud} consists of partitions $(f_u)_{u\geq 0}$ such that 
\begin{equation}\label{eq:cdu}
\begin{cases}
f_0 =0,\\
f_1\leq i-1,\\
\text{for all }u\geq 1, \ f_u+f_{u+1}\leq r-1,\\
\text{for all }u\geq 1, \ f_u+f_{u+1}=r-1 \Rightarrow u f_u+(u+1) f_{u+1} \equiv i-1 \mod 2.
\end{cases}
\end{equation}
The set $\tilde{\mathcal{U}}_{i,r}$ described in Corollary~\ref{coro:newbressoud} consists of partitions $(f_u)_{u\geq 0}$ such that 
\begin{equation}\label{eq:cdunew}
\begin{cases}
f_0 =0,\\
f_1\leq i-1,\\
\text{for all }u\geq 1, \ f_u+f_{u+1}\leq r-1,\\
\text{for all }u\geq 1, \ f_u+f_{u+1}=r-1 \Rightarrow u f_u+(u+1) f_{u+1} \equiv i \mod 2.
\end{cases}
\end{equation}
\end{Proposition}
\begin{proof}
The description \eqref{eq:cdt} follows from Proposition \ref{prop:cdf} (1) with $d=2$, $m=r-1$, while \eqref{eq:cdu} (resp. \eqref{eq:cdunew}) follows from Proposition \ref{prop:cdf} (2) with $d=2$, $m=r-1$ and $(P)$ the property of being congruent to $i-1 \mod 2$ (resp. $i \mod 2$).
\end{proof}

We define several related sets of partitions in terms of their multiplicity sequences $(f_u)_{u\geq 0}$, where now $0$ parts are allowed. 

\begin{Definition}\label{def:ABB}
Let $r$ and $i$ be integers such that $r\geq 2$ and $0\leq i \leq r-1$.
\begin{itemize}
\item Let $\mathcal{A}_{i,r}$ be the set of partition $(f_u)_{u\geq 0}$ such that $f_0\leq i$ and $f_u+f_{u+1}\leq r-1$ for all $u$. 
\item Let $\mathcal{B}_{i,r}$ be the set of partitions $(f_u)_{u\geq 0}$ of $\mathcal{A}_{i,r}$ such that, for all $u$, $f_u+f_{u+1}=r-1$ only if
 $$u f_u+ (u+1) f_{u+1} \equiv r-1-i \mod 2.$$
\item Let $\tilde{\mathcal{B}}_{i,r}$ be the set of partitions $(f_u)_{u\geq 0}$ of $\mathcal{A}_{i,r}$ such that, for all $j$,  $f_u+f_{u+1}=r-1$ only if
 $$uf_u+ (u+1) f_{u+1} \equiv r-i \mod 2.$$
\end{itemize}
We choose the convention that $\mathcal{A}_{-1,r}=\mathcal{B}_{-1,r}=\tilde{\mathcal{B}}_{-1,r}=\mathcal{T}_{0,r}=\mathcal{U}_{0,r}=\tilde{\mathcal{U}}_{0,r}=\emptyset$.
\end{Definition}

Note that from our combinatorial point of view, as parity conditions always come in pairs, the set $\tilde{\mathcal{U}}_{i,r}$ (resp. $\tilde{\mathcal{B}}_{i,r}$) arises in a natural way together with $\mathcal{U}_{i,r}$ (resp. $\mathcal{B}_{i,r}$). This explains our discovery of Corollary~\ref{coro:newbressoud} and~\eqref{fijform}.

Also observe that $\mathcal{A}_{r-1,r}=\mathcal{A}_{r}$ defined in~\eqref{defAr}, and that for all $0 \leq i \leq r-1$, $\mathcal{A}_{i-1,r}\subset\mathcal{A}_{i,r}$, $\mathcal{B}_{i-1,r}=\mathcal{A}_{i-1,r}\cap\tilde{\mathcal{B}}_{i,r}$, and $\tilde{\mathcal{B}}_{i-1,r} = \mathcal{A}_{i-1,r}\cap \mathcal{B}_{i,r}$.
Similarly, $\mathcal{T}_{i,r}\subset\mathcal{T}_{i+1,r}$, $\mathcal{U}_{i,r}=\mathcal{T}_{i,r}\cap\tilde{\mathcal{U}}_{i+1,r}$, and $\tilde{\mathcal{U}}_{i,r} = \mathcal{T}_{i,r}\cap \mathcal{U}_{i+1,r}$.
The following results give a precise description of the relations between the sets of Definition \ref{def:ABB} and the sets $\mathcal{T}_{i,r}$, $\mathcal{U}_{i,r}$, $\tilde{\mathcal{U}}_{i,r}$.

\begin{Lemma}\label{lem:ajoutzero}
For all integers $i,r$ such that $r\geq 2$ and $0\leq i\leq r-1$, the map $(f_u)_{u\geq 0} \mapsto (f_u)_{u\geq 1}$ defines a weight-preserving bijection
\begin{enumerate}
\item from $\mathcal{A}_{i,r}\setminus \mathcal{A}_{i-1,r}$ to $\mathcal{T}_{r-i,r}$,
\item from $\mathcal{B}_{i,r}\setminus \tilde{\mathcal{B}}_{i-1,r}$ to $\mathcal{U}_{r-i,r}$,
\item from $\tilde{\mathcal{B}}_{i,r}\setminus \mathcal{B}_{i-1,r}$ to $\mathcal{U}_{r-i-1,r}$,
\end{enumerate}
with inverse bijection given by
$(f_1, f_2, \dots) \mapsto (i,f_1,f_2,\dots)$.
\end{Lemma}

\begin{proof} As $\sum_{u\geq 0}uf_u = \sum_{u\geq 1}uf_u$, the map $(f_u)_{u\geq 0} \mapsto (f_u)_{u\geq 1}$ is weight-preserving.
\begin{enumerate}
\item[(1)] For all $(f_u)_{u\geq 0}\in \mathcal{A}_{i,r}\setminus \mathcal{A}_{i-1,r}$, we have $f_0=i$ and $f_u+f_{u+1}\leq r-1$ for all $u\geq 0$. Therefore $f_1\leq r-1-i$ and $f_u+f_{u+1}\leq r-1$ for all $u\geq 1$, and by~\eqref{eq:cdt}, the partition $(f_u)_{u\geq 1}$ belongs to $\mathcal{T}_{r-i,r}$. Conversely, for all $(f_u)_{u\geq 1} \in \mathcal{T}_{r-i,r}$, by setting $f_0=i$, we obtain that $(f_u)_{u\geq 0} \in \mathcal{A}_{i,r}\setminus \mathcal{A}_{i-1,r}$.
\item[(2)] The proof is similar to (1), using~\eqref{eq:cdu} instead of~\eqref{eq:cdt}.
\item[(3)] For all $(f_u)_{u\geq 0}\in \tilde{\mathcal{B}}_{i,r}\setminus \mathcal{B}_{i-1,r}$, we have $f_0=i$, and $f_u+f_{u+1}\leq r-1$ with equality only if $uf_u+(u+1)f_{u+1}\equiv r-i \mod 2$ for all $u\geq 0$. Thus $f_1\neq r-1-i$, and for all $(f_u)_{u\geq 0}\in \tilde{\mathcal{B}}_{i,r}\setminus \mathcal{B}_{i-1,r}$, we have $f_1\leq r-2-i$, and $f_u+f_{u+1}\leq r-1$ with equality only if $uf_u+(u+1)f_{u+1}\equiv r-i \mod 2$ for all $u\geq 1$. Hence, by ~\eqref{eq:cdu}, the partition $(f_u)_{u\geq 1}$ belongs to $\mathcal{U}_{r-i-1,r}$. Conversely, for all $(f_u)_{u\geq 1} \in \mathcal{U}_{r-i-1,r}$, by setting $f_0=i$, we obtain that $(f_u)_{u\geq 0}\in \tilde{\mathcal{B}}_{i,r}\setminus \mathcal{B}_{i-1,r}$.
\end{enumerate}
\end{proof}

The next result will be useful for proving Corollary~\ref{coro:newbressoud}.

\begin{Lemma}\label{lem:ajoutun}
For all integers $i,r$ such that $r\geq 2$ and $1\leq i\leq r-1$, the map $(f_u)_{u\geq 0} \mapsto (f_0,f_1-1,f_2,f_3,\ldots)$ defines a bijection from
$\mathcal{U}_{i+1,r}\setminus \tilde{\mathcal{U}}_{i,r}$ to $\tilde{\mathcal{U}}_{i,r}\setminus \mathcal{U}_{i-1,r}$, which decreases the weight by $1$.
\end{Lemma}

\begin{proof}
Note that by~\eqref{eq:cdu} and~\eqref{eq:cdunew}, the set $\mathcal{U}_{i+1,r}\setminus \tilde{\mathcal{U}}_{i,r}$ consists of the partitions of $\mathcal{U}_{i+1,r}$ such that $f_1=i$, while $\tilde{\mathcal{U}}_{i,r}\setminus \mathcal{U}_{i-1,r}$ consists of the partitions of $\mathcal{U}_{i+1,r}$ such that $f_1=i-1$. Hence, by the uniqueness of the multiplicity sequence, the map $(f_u)_{u\geq 0} \mapsto (f_0,f_1-1,f_2,f_3,\ldots)$ defines an injection from
$\mathcal{U}_{i+1,r}\setminus \tilde{\mathcal{U}}_{i,r}$ to $\tilde{\mathcal{U}}_{i,r}\setminus \mathcal{U}_{i-1,r}$.

Conversely, let $(f_u)_{u\geq 0}$ be a partition in $\tilde{\mathcal{U}}_{i,r}\setminus \mathcal{U}_{i-1,r}$. It is therefore a partition of $\mathcal{U}_{i+1,r}$ such that $f_1=i-1$. In particular $f_1+2f_2 \equiv i-1 \mod 2$, thus by definition of $\mathcal{U}_{i+1,r}$, we cannot have $f_1+f_2=r-1$. Hence, $f_1+f_2\leq r-2$. Then, by adding a part $1$ to the partition, we obtain a new partition with  multiplicity sequence $(f_0,f_1+1,f_2,f_3,\ldots)$ in $\mathcal{U}_{i+1,r}$ such that $f_1+1 = i$.  The map $(f_u)_{u\geq 0} \mapsto (f_0,f_1-1,f_2,f_3,\ldots)$ thus defines a surjection from
$\mathcal{U}_{i+1,r}\setminus \tilde{\mathcal{U}}_{i,r}$ to $\tilde{\mathcal{U}}_{i,r}\setminus \mathcal{U}_{i-1,r}$, and we can conclude.
\end{proof}

Provided Theorems~\ref{th:Gordon} and~\ref{th:Bressoud}, and Lemmas~\ref{lem:ajoutzero} and~\ref{lem:ajoutun}, a natural combinatorial description emerges for the right-hand sides of identities~~\eqref{eq:AGri},~\eqref{eq:Bri},~\eqref{eq:fij}, and~\eqref{Br3.3form}--\eqref{fijform3}, in terms of generating functions of sets related to $\mathcal{A}_{i,r}, \mathcal{B}_{i,r},\tilde{\mathcal{B}}_{i,r}, \mathcal{T}_{i,r}, \mathcal{U}_{i,r}$, and $\tilde{\mathcal{U}}_{i,r}$. Note that the right-hand sides of~\eqref{AGribis}--\eqref{Bri2bis} correspond to the ones of~\eqref{eq:AGri},~\eqref{eq:Bri}, and~\eqref{eq:fij}, respectively. The right-hand sides of~\eqref{Br3.3bis}--\eqref{BPbisbis} are the ones of~\eqref{Br3.3form}--\eqref{fijform3}.

\begin{Proposition}\label{prop:prod}
For all integer $r\geq 2$, we have 
\begin{align}
\label{AGribis}
&\sum_{\lambda \in \mathcal{T}_{i,r}} q^{|\lambda|} = \frac{(q^{2r+1},q^{i},q^{2r-i+1};q^{2r+1})_\infty}{(q)_\infty}&\text{for  } 1\leq i\leq r,\\
\label{Bribis}
&\sum_{\lambda \in \mathcal{U}_{i,r}} q^{|\lambda|} = \frac{(q^{2r},q^{i},q^{2r-i};q^{2r})_\infty}{(q)_\infty} &\text{for  } 1\leq i\leq r,\\
\label{Bri2bis}
&(1+q)\sum_{\lambda \in \tilde{\mathcal{U}}_{i,r}} q^{|\lambda|} =\frac{(q^{2r},q^{i+1},q^{2r-i-1};q^{2r})_\infty+q(q^{2r},q^{i-1},q^{2r-i+1};q^{2r})_\infty}{(q)_\infty}&\text{for  } 1\leq i\leq r,\\
\label{Br3.3bis}
&\sum_{\lambda \in \mathcal{A}_{i,r}} q^{|\lambda|} = \sum_{k=0}^{i}\frac{(q^{2r+1},q^{r-i+k},q^{r+i-k+1};q^{2r+1})_\infty}{(q)_\infty}&\text{for  } 0\leq i\leq r-1,\\
\label{Br3.5bis}
&\sum_{\lambda \in \mathcal{B}_{i,r}} q^{|\lambda|} = \sum_{k=0}^{i}\frac{(q^{2r},q^{r-i+2k},q^{r+i-2k};q^{2r})_\infty}{(q)_\infty} &\text{for    } 0\leq i\leq r-1,\\
\label{fijbis}
&\sum_{\lambda \in \tilde{\mathcal{B}}_{i,r}} q^{|\lambda|} =\sum_{k=0}^{i}\frac{(q^{2r},q^{r-i+2k-1},q^{r+i-2k+1};q^{2r})_\infty}{(q)_\infty} &\text{for  } 0\leq i\leq r-1,\\
\label{AGPbis}
&\sum_{\lambda \in \mathcal{A}_{i,r}\setminus \mathcal{A}_{i-1,r}} q^{|\lambda|} = \frac{(q^{2r+1},q^{r-i},q^{r+i+1};q^{2r+1})_\infty}{(q)_\infty}&\text{for  } 0\leq i\leq r-1,\\
\label{BPbis}
&\sum_{\lambda \in \mathcal{B}_{i,r}\setminus \mathcal{B}_{i-2,r}} q^{|\lambda|} = 2\frac{(q^{2r},q^{r-i},q^{r+i};q^{2r})_\infty}{(q)_\infty}&\text{for  } 1\leq i\leq r-1,\\
\label{isaacbis}
&\sum_{\lambda \in \mathcal{B}_{i,r}\setminus \tilde{\mathcal{B}}_{i-1,r}} q^{|\lambda|} = \frac{(q^{2r},q^{r-i},q^{r+i};q^{2r})_\infty}{(q)_\infty}&\text{for  } 0\leq i\leq r-1,\\
\label{isaacbisbis}
&\sum_{\lambda \in \tilde{\mathcal{B}}_{i,r}\setminus \mathcal{B}_{i-1,r}} q^{|\lambda|} = \frac{(q^{2r},q^{r-i-1},q^{r+i+1};q^{2r})_\infty}{(q)_\infty}&\text{for  } 0\leq i\leq r-1,\\
\label{BPbisbis}
&\sum_{\lambda \in \tilde{\mathcal{B}}_{i,r}\setminus \tilde{\mathcal{B}}_{i-2,r}} q^{|\lambda|} = \frac{(q^{2r},q^{r-i-1},q^{r+i+1};q^{2r})_\infty+(q^{2r},q^{r-i+1},q^{r+i-1};q^{2r})_\infty}{(q)_\infty}&\text{for  } 1\leq i\leq r-1.
\end{align}

\end{Proposition}

\begin{proof}
Identity~\eqref{AGribis} (resp.~\eqref{Bribis}) is a direct consequence of Theorem~\ref{th:Gordon} (resp. Theorem~\ref{th:Bressoud}) and~\eqref{eq:cdt} (resp.~\eqref{eq:cdu}) in Proposition~\ref{prop:TUfreq}. 

Moreover, for $1\leq i\leq r-1$, we have $\mathcal{U}_{i-1,r}\subset \tilde{\mathcal{U}}_{i,r}\subset \mathcal{U}_{i+1,r}$, and
\begin{align*}
(1+q)\sum_{\lambda \in \tilde{\mathcal{U}}_{i,r}} q^{|\lambda|} &= \sum_{\lambda \in \tilde{\mathcal{U}}_{i,r}} q^{|\lambda|} + \sum_{\lambda \in \tilde{\mathcal{U}}_{i,r}\setminus \mathcal{U}_{i-1,r}} q^{|\lambda|+1} + q\sum_{\lambda \in \mathcal{U}_{i-1,r}} q^{|\lambda|}\\
&=\sum_{\lambda \in \tilde{\mathcal{U}}_{i,r}} q^{|\lambda|} + \sum_{\lambda \in \mathcal{U}_{i+1,r}\setminus \tilde{\mathcal{U}}_{i,r}} q^{|\lambda|} + q\sum_{\lambda \in \mathcal{U}_{i-1,r}} q^{|\lambda|}&\text{by Lemma  ~\ref{lem:ajoutun}}\\
&=\sum_{\lambda \in \mathcal{U}_{i+1,r}} q^{|\lambda|}+q\sum_{\lambda \in \mathcal{U}_{i-1,r}} q^{|\lambda|}.
\end{align*}
Using~\eqref{Bribis}, we obtain~\eqref{Bri2bis} for $1\leq i\leq r-1$. Observe that $\tilde{\mathcal{U}}_{r,r} = \mathcal{U}_{r-1,r}$ by~\eqref{eq:cdu} and~\eqref{eq:cdunew}, so that~\eqref{Bri2bis} holds for $i=r$.

Formula~\eqref{AGPbis} comes from Lemma~\ref{lem:ajoutzero}~(1) and Theorem~\ref{th:Gordon}. Noting that
$$\mathcal{A}_{i,r} = \bigsqcup_{k=0}^i  (\mathcal{A}_{i-k,r}\setminus \mathcal{A}_{i-1-k,r}),$$
we deduce~\eqref{Br3.3bis}. Lemma~\ref{lem:ajoutzero}~(2) and Theorem~\ref{th:Bressoud} yield~\eqref{isaacbis}. By Lemma~\ref{lem:ajoutzero} (3) and Theorem~\ref{th:Bressoud}, we derive~\eqref{isaacbisbis}. Using~\eqref{isaacbis}, \eqref{isaacbisbis}, and the equality $\mathcal{B}_{i,r}\setminus \mathcal{B}_{i-2,r} = (\mathcal{B}_{i,r}\setminus \tilde{\mathcal{B}}_{i-1,r})\sqcup (\tilde{\mathcal{B}}_{i-1,r}\setminus \mathcal{B}_{i-2,r} )$, we deduce~\eqref{BPbis}. 

To prove~\eqref{Br3.5bis}, first observe that  
$$\mathcal{B}_{i,r} = \left(\bigsqcup_{k=0}^{\lfloor i/2\rfloor}\mathcal{B}_{i-2k,r}\setminus \tilde{\mathcal{B}}_{i-2k-1,r}\right)\sqcup \left(\bigsqcup_{k=0}^{\lfloor (i-1)/2\rfloor}\tilde{\mathcal{B}}_{i-2k-1,r}\setminus \mathcal{B}_{i-2k-2,r}\right).$$
By ~\eqref{isaacbisbis}, we have
$$\sum_{\lambda \in \tilde{\mathcal{B}}_{i-2k-1,r}\setminus \mathcal{B}_{i-2k-2,r}} q^{|\lambda|}= \frac{(q^{2r},q^{r-i+2k},q^{r+i-2k};q^{2r})_\infty}{(q)_\infty} = \frac{(q^{2r},q^{r-i+2(i-k)},q^{r+i-2(i-k)};q^{2r})_\infty}{(q)_\infty}.$$
Hence, using~\eqref{isaacbis}, we deduce
$$\sum_{\lambda \in \mathcal{B}_{i,r}} q^{|\lambda|} = \sum_{k=0}^{\lfloor i/2\rfloor} \frac{(q^{2r},q^{r-i+2k},q^{r+i-2k};q^{2r})_\infty}{(q)_\infty} + \sum_{k=i-\lfloor (i-1)/2\rfloor}^i \frac{(q^{2r},q^{r-i+2k},q^{r+i-2k};q^{2r})_\infty}{(q)_\infty}.$$
Since $i=\lfloor i/2\rfloor+\lceil i/2\rceil = \lfloor i/2\rfloor + \lfloor (i-1)/2\rfloor+1$, the integers $\lfloor i/2\rfloor$ and $i-\lfloor(i-1)/2\rfloor$ are consecutive and ~\eqref{Br3.5bis} holds.

Next, by ~\eqref{Br3.5bis} we have
$$\sum_{\lambda \in \mathcal{B}_{i-1,r}} q^{|\lambda|} = \sum_{k=0}^{i-1}\frac{(q^{2r},q^{r-i+1+2k},q^{r+i-1-2k};q^{2r})_\infty}{(q)_\infty}=\sum_{k=1}^{i}\frac{(q^{2r},q^{r-i-1+2k},q^{r+i+1-2k};q^{2r})_\infty}{(q)_\infty},$$
therefore we derive~\eqref{fijbis} using $\tilde{\mathcal{B}}_{i,r} = (\tilde{\mathcal{B}}_{i,r}\setminus \mathcal{B}_{i-1,r})\sqcup \mathcal{B}_{i-1,r}$ and~\eqref{isaacbisbis}. 

Finally, writing $\tilde{\mathcal{B}}_{i,r}\setminus \tilde{\mathcal{B}}_{i-2,r}=(\tilde{\mathcal{B}}_{i,r}\setminus \mathcal{B}_{i-1,r})\sqcup (\mathcal{B}_{i-1,r}\setminus \tilde{\mathcal{B}}_{i-2,r})$, and using~\eqref{isaacbisbis} and~\eqref{isaacbis}, we derive~\eqref{BPbisbis}.

\end{proof}

\subsection{Combinatorial description of the left-hand (multisum) sides}

Let $r\geq2$ be an integer, let $s_1\geq \cdots \geq s_{r-1}\geq 0$ be integers, and set $s_0=\infty$ and $s_r= 0$. On the multisum sides of~\eqref{eq:AGri},~\eqref{eq:Bri},~\eqref{eq:fij}, and~\eqref{Br3.3form}--\eqref{fijform3}, the summands can all be factorized by $q^{s_1^2+\cdots+s_{r-1}^2 -s_{1}-\cdots-s_{r-1}}$, which does not depend on $i$. Hence, for generating all these multisum sides, we first need a partition whose weight is $s_1^2+\cdots+s_{r-1}^2 -s_{1}-\cdots-s_{r-1}$: this is $\mu(s_1,\ldots,s_{r-1})$ from Definition~\ref{def:mu}. Indeed,  for all $j\in \{0,\ldots,r-1\}$ and all $u\in\{s_{j+1},\ldots,s_j-1\}$, we have $f_{2u}+f_{2u+1}=j$. Therefore the number of parts of $\mu(s_1,\ldots,s_{r-1})$ is
$$\sum_{u\geq 0} f_u = (r-1) s_{r-1}+\sum_{j=1}^{r-2} j  (s_{j}-s_{j+1})=s_{1}+\cdots+s_{r-1},$$
and its weight is
$$\sum_{u\geq 0}uf_u = \sum_{j=1}^{r-1} \sum_{u=s_{j+1}}^{s_j-1}j\cdot 2u= \sum_{j=1}^{r-1} \sum_{u=s_{j+1}}^{s_j-1}\sum_{k=1}^j 2u=\sum_{k=1}^{r-1} \sum_{j=k}^{r-1}\sum_{u=s_{j+1}}^{s_{j}-1} 2u =\sum_{k=1}^{r-1}\sum_{u=0}^{s_{k}-1} 2u,
$$
which gives
\begin{equation}\label{eq:weightmu}
|\mu(s_1,\ldots,s_{r-1})|= s_1^2+\cdots+s_{r-1}^2 -s_{1}-\cdots-s_{r-1}.
\end{equation}

We now define $(r-1)$-tuples of partitions in order to explain the $q$-Pochhammer symbols in the denominator of the multisum sides of our identities.
 
\begin{Definition}\label{def:LHSsets}
Recall from Definition \ref{def:P} that $\mathcal{P}(s_1,\ldots,s_{r-1})$ is the set of sequences $\lambda=(\lambda_0,\ldots,\lambda_{s_1-1})$ of non-negative integers such that for all $j\in \{1,\ldots,r-1\}$, the sequence $(\lambda_{s_{j+1}},\ldots,\lambda_{s_j-1})$ is non-decreasing.
Let $|\lambda|:=\lambda_0+\cdots+\lambda_{s_1-1}$ denote the weight of $\lambda$. For all $0\leq i\leq r-1$, we now define the following subsets of $\mathcal{P}(s_1,\ldots,s_{r-1})$:
\begin{itemize}
\item Let $\mathcal{P}_{i,r}(s_1,\ldots,s_{r-1})$ be the subset of $\mathcal{P}(s_1,\ldots,s_{r-1})$ whose elements $\lambda=(\lambda_0,\ldots,\lambda_{s_1-1})$ satisfy: $\lambda_{s_{j+1}}\geq j-i$ for all $1\leq j\leq r-1$.
\item Let $\mathcal{R}_{i,r}(s_1,\ldots,s_{r-1})$ be the subset of $\mathcal{P}_{i,r}(s_1,\ldots,s_{r-1})$ whose elements $\lambda=(\lambda_0,\ldots,\lambda_{s_1-1})$ satisfy: $\lambda_{0},\ldots,\lambda_{s_{r-1}-1}$ have the same parity as $r-1-i$.
\item Let $\tilde{\mathcal{R}}_{i,r}(s_1,\ldots,s_{r-1})$ be the subset of $\mathcal{P}_{i,r}(s_1,\ldots,s_{r-1})$  whose elements $\lambda=(\lambda_0,\ldots,\lambda_{s_1-1})$ satisfy: $\lambda_{0},\ldots,\lambda_{s_{r-1}-1}$ have the same parity as $r-i$.
\item Let $\mathcal{Q}_{i,r}(s_1,\ldots,s_{r-1})$ be the subset of $\mathcal{P}(s_1,\ldots,s_{r-1})$  whose elements $\lambda=(\lambda_0,\ldots,\lambda_{s_1-1})$ satisfy: $\lambda_{s_{j+1}}\geq j+\max\{j-i,0\}$ for all $1\leq j\leq r-1$.
\item Let $\mathcal{S}_{i,r}(s_1,\ldots,s_{r-1})$ be the subset of $\mathcal{Q}_{i,r}(s_1,\ldots,s_{r-1})$ whose elements $\lambda=(\lambda_0,\ldots,\lambda_{s_1-1})$ satisfy: $\lambda_{0},\ldots,\lambda_{s_{r-1}-1}$ have the same parity as $i$.
\item Let $\tilde{\mathcal{S}}_{i,r}(s_1,\ldots,s_{r-1})$ be the subset of $\mathcal{Q}_{i,r}(s_1,\ldots,s_{r-1})$  whose elements $\lambda=(\lambda_0,\ldots,\lambda_{s_1-1})$ satisfy: $\lambda_{0},\ldots,\lambda_{s_{r-1}-1}$ have the same parity as $i-1$.
\end{itemize}
Finally, for $\mathcal{L}\in\{\mathcal{P},\mathcal{Q},\mathcal{R},\tilde{\mathcal{R}},\mathcal{S},\tilde{\mathcal{S}}\}$, define
$$\mathcal{L}_{i,r} := \bigsqcup_{s_1\geq \cdots \geq s_{r-1}\geq 0} \{\mu(s_1,\ldots,s_{r-1})\}\times \mathcal{L}_{i,r}(s_1,\ldots,s_{r-1}),$$
and for all $(\mu,\lambda)\in \mathcal{L}_{i,r}$, define its weight as $|(\mu,\lambda)|=|\mu|+|\lambda|$.
\end{Definition}


Note that $\mathcal{P}_r$ defined in the introduction is equal to $\mathcal{P}_{r-1,r}$. 

The multisum sides of identities~\eqref{eq:AGri},~\eqref{eq:Bri},~\eqref{eq:fij}, and~\eqref{Br3.3form}--\eqref{fijform3} can be written as generating functions for sets expressed in terms of $\mathcal{P}_{i,r}$, $\mathcal{Q}_{i,r}$, $\mathcal{R}_{i,r}$, $\tilde{\mathcal{R}}_{i,r}$, $\mathcal{S}_{i,r}$, and $\tilde{\mathcal{S}}_{i,r}$.
 In particular, note that in the result below, the right-hand sides of~\eqref{AGriter}--\eqref{Bri2ter} correspond to the multisum sides of~\eqref{eq:AGri},~\eqref{eq:Bri}, and~\eqref{eq:fij}, respectively. The right-hand sides of~\eqref{Br3.3ter}--\eqref{fij4} are the multisum sides of~\eqref{Br3.3form}--\eqref{fijform3}.

\begin{Proposition}\label{prop:multisum}
For all integers $r$ such that $r\geq 2$, we have 
\begin{align}
\label{AGriter}
&\sum_{(\mu,\lambda) \in \mathcal{Q}_{i-1,r}} q^{|(\mu,\lambda)|} = \sum_{s_1\geq \dots\geq s_{r-1}\geq 0}\frac{q^{s_1^2+\dots+s_{r-1}^2 +s_{i}+\cdots+s_{r-1}}}{(q)_{s_1-s_2}\dots(q)_{s_{r-2}-s_{r-1}}(q)_{s_{r-1}}}&\text{for  } 1\leq i\leq r,\\
\label{Briter}
&\sum_{(\mu,\lambda) \in \mathcal{S}_{i-1,r}} q^{|(\mu,\lambda)|} = \sum_{s_1\geq \dots\geq s_{r-1}\geq 0}\frac{q^{s_1^2+\dots+s_{r-1}^2 +s_{i}+\cdots+s_{r-1}}}{(q)_{s_1-s_2}\dots(q)_{s_{r-2}-s_{r-1}}(q^2;q^2)_{s_{r-1}}}&\text{for  } 1\leq i\leq r,\\
\label{Bri2ter}
&\sum_{(\mu,\lambda) \in \tilde{\mathcal{S}}_{i-1,r}} q^{|(\mu,\lambda)|} = \sum_{s_1\geq \dots\geq s_{r-1}\geq 0}\frac{q^{s_1^2+\dots+s_{r-1}^2 +s_{i}+\cdots+s_{r-2}+2s_{r-1}}}{(q)_{s_1-s_2}\dots(q)_{s_{r-2}-s_{r-1}}(q^2;q^2)_{s_{r-1}}}&\text{for  } 1\leq i\leq r,\\
\label{Br3.3ter}
&\sum_{(\mu,\lambda) \in \mathcal{P}_{i,r}} q^{|(\mu,\lambda)|} = \sum_{s_1\geq \dots\geq s_{r-1}\geq 0}\frac{q^{s_1^2+\dots+s_{r-1}^2 -s_{1}-\cdots-s_{i}}}{(q)_{s_1-s_2}\dots(q)_{s_{r-2}-s_{r-1}}(q)_{s_{r-1}}}&\text{for  } 0\leq i\leq r-1,\\
\label{Br3.5ter}
&\sum_{(\mu,\lambda) \in \mathcal{R}_{i,r}} q^{|(\mu,\lambda)|} = \sum_{s_1\geq \dots\geq s_{r-1}\geq 0}\frac{q^{s_1^2+\dots+s_{r-1}^2 -s_{1}-\cdots-s_{i}}}{(q)_{s_1-s_2}\dots(q)_{s_{r-2}-s_{r-1}}(q^2;q^2)_{s_{r-1}}}&\text{for  } 0\leq i\leq r-1,\\
\label{fijter}
&\sum_{(\mu,\lambda) \in \tilde{\mathcal{R}}_{i,r}} q^{|(\mu,\lambda)|} = \sum_{s_1\geq \dots\geq s_{r-1}\geq 0}\frac{q^{s_1^2+\dots+s_{r-1}^2 -s_{1}-\cdots-s_{i}+s_{r-1}}}{(q)_{s_1-s_2}\dots(q)_{s_{r-2}-s_{r-1}}(q^2;q^2)_{s_{r-1}}}&\text{for  } 0\leq i\leq r-1,\\
\label{AGPter}
&\sum_{(\mu,\lambda) \in \mathcal{P}_{i,r}\setminus \mathcal{P}_{i-1,r}} q^{|(\mu,\lambda)|} = \sum_{s_1\geq \dots\geq s_{r-1}\geq 0}\frac{q^{s_1^2+\dots+s_{r-1}^2 -s_{1}-\cdots-s_{i}}(1-q^{s_{i}})}{(q)_{s_1-s_2}\dots(q)_{s_{r-2}-s_{r-1}}(q)_{s_{r-1}}}&\text{for  } 1\leq i\leq r-1,\\
\label{BPter}
&\sum_{(\mu,\lambda) \in \mathcal{R}_{i,r}\setminus \mathcal{R}_{i-2,r}} q^{|(\mu,\lambda)|} = \sum_{s_1\geq \dots\geq s_{r-1}\geq 0}\frac{q^{s_1^2+\dots+s_{r-1}^2 -s_{1}-\cdots-s_{i}}(1-q^{s_{i}+s_{i-1}})}{(q)_{s_1-s_2}\dots(q)_{s_{r-2}-s_{r-1}}(q^2;q^2)_{s_{r-1}}}&\text{for  } 2\leq i\leq r-1,\\
\label{isaacter}
&\sum_{(\mu,\lambda) \in \mathcal{R}_{i,r}\setminus \tilde{\mathcal{R}}_{i-1,r}} q^{|(\mu,\lambda)|} = \sum_{s_1\geq \dots\geq s_{r-1}\geq 0}\frac{q^{s_1^2+\dots+s_{r-1}^2 -s_{1}-\cdots-s_{i}}(1-q^{s_{i}+s_{r-1}})}{(q)_{s_1-s_2}\dots(q)_{s_{r-2}-s_{r-1}}(q^2;q^2)_{s_{r-1}}}&\text{for  } 1\leq i\leq r-1,\\
\label{fijter2}
&\sum_{(\mu,\lambda) \in \tilde{\mathcal{R}}_{i,r}\setminus \mathcal{R}_{i-1,r}} q^{|(\mu,\lambda)|} = \sum_{s_1\geq \dots\geq s_{r-1}\geq 0}\frac{q^{s_1^2+\dots+s_{r-1}^2 -s_{1}-\cdots-s_{i}}(q^{s_{r-1}}-q^{s_i})}{(q)_{s_1-s_2}\dots(q)_{s_{r-2}-s_{r-1}}(q^2;q^2)_{s_{r-1}}}&\text{for  } 1\leq i\leq r-1,\\
\label{fij4}
&\sum_{(\mu,\lambda) \in \tilde{\mathcal{R}}_{i,r}\setminus \tilde{\mathcal{R}}_{i-2,r}} q^{|(\mu,\lambda)|} = \sum_{s_1\geq \dots\geq s_{r-1}\geq 0}\frac{q^{s_1^2+\dots+s_{r-1}^2 -s_{1}-\cdots-s_{i}+s_{r-1}}(1-q^{s_{i}+s_{i-1}})}{(q)_{s_1-s_2}\dots(q)_{s_{r-2}-s_{r-1}}(q^2;q^2)_{s_{r-1}}}&\text{for  } 2\leq i\leq r-1.
\end{align}
\end{Proposition}

\begin{proof}
Recall that for all integers $k,l,m$ with $k,l\geq 0$ and $m\geq 1$, the generating function for partitions into $k$ non-zero parts $\geq l$ and congruent to $l \mod m$ is given by $q^{kl}/(q^m;q^m)_k$, and that zero parts do not contribute to generating functions. By computing the generating functions for partitions $\lambda^{(j)}=(\lambda_{s_{j}-1},\ldots,\lambda_{s_{j+1}})$ such that $\lambda=(\lambda_0,\ldots,\lambda_{s_1-1})$ belongs to $\mathcal{P}(s_1,\ldots,s_{r-1})$ or its subsets from Definition~\ref{def:LHSsets}, we deduce the following:

\begin{align*}
\sum_{\lambda \in \mathcal{Q}_{i,r}(s_1,\ldots,s_{r-1})} q^{|\lambda|} &= \prod_{j=1}^{i}\frac{q^{j(s_j-s_{j+1})}}{(q)_{s_j-s_{j+1}}} \times\prod_{j=i+1}^{r-1}\frac{q^{(2j-i)(s_{j}-s_{j+1})}}{(q)_{s_{j}-s_{j+1}}}= \frac{q^{s_1+\cdots+s_i+2s_{i+1}+\cdots+2s_{r-1}}}{(q)_{s_1-s_2}\dots(q)_{s_{r-2}-s_{r-1}}(q)_{s_{r-1}}},\\
\sum_{\lambda \in \mathcal{S}_{i,r}(s_1,\ldots,s_{r-1})} q^{|\lambda|} &= \prod_{j=1}^{\min\{i,r-2\}}\frac{q^{j(s_j-s_{j+1})}}{(q)_{s_j-s_{j+1}}} \times\prod_{j=i+1}^{r-2}\frac{q^{(2j-i)(s_{j}-s_{j+1})}}{(q)_{s_{j}-s_{j+1}}}\times \frac{q^{(2r-2-i)s_{r-1}}}{(q^2;q^2)_{s_{r-1}}}\\&= \frac{q^{s_1+\cdots+s_i+2s_{i+1}+\cdots+2s_{r-1}}}{(q)_{s_1-s_2}\dots(q)_{s_{r-2}-s_{r-1}}(q^2;q^2)_{s_{r-1}}},\\
\sum_{\lambda \in \tilde{\mathcal{S}}_{i,r}(s_1,\ldots,s_{r-1})} q^{|\lambda|} &= \prod_{j=1}^{\min\{i,r-2\}}\frac{q^{j(s_j-s_{j+1})}}{(q)_{s_j-s_{j+1}}} \times\prod_{j=i+1}^{r-2}\frac{q^{(2j-i)(s_{j}-s_{j+1})}}{(q)_{s_{j}-s_{j+1}}}\times \frac{q^{(2r-1-i)s_{r-1}}}{(q^2;q^2)_{s_{r-1}}}\\&= \frac{q^{(s_1+\cdots+s_i+2s_{i+1}+\cdots+2s_{r-1})+s_{r-1}}}{(q)_{s_1-s_2}\dots(q)_{s_{r-2}-s_{r-1}}(q^2;q^2)_{s_{r-1}}},\\
\sum_{\lambda \in \mathcal{P}_{i,r}(s_1,\ldots,s_{r-1})} q^{|\lambda|} &= \prod_{j=1}^{i}\frac{1}{(q)_{s_j-s_{j+1}}} \times\prod_{j=i+1}^{r-1}\frac{q^{(j-i)(s_{j}-s_{j+1})}}{(q)_{s_{j}-s_{j+1}}}= \frac{q^{s_{i+1}+\cdots+s_{r-1}}}{(q)_{s_1-s_2}\dots(q)_{s_{r-2}-s_{r-1}}(q)_{s_{r-1}}},\\
\sum_{\lambda \in \mathcal{R}_{i,r}(s_1,\ldots,s_{r-1})} q^{|\lambda|} &= \prod_{j=1}^{\min\{i,r-2\}}\frac{1}{(q)_{s_j-s_{j+1}}} \times\prod_{j=i+1}^{r-2}\frac{q^{(j-i)(s_{j}-s_{j+1})}}{(q)_{s_{j}-s_{j+1}}}\times \frac{q^{(r-1-i)s_{r-1}}}{(q^2;q^2)_{s_{r-1}}}\\&= \frac{q^{s_{i+1}+\cdots+s_{r-1}}}{(q)_{s_1-s_2}\dots(q)_{s_{r-2}-s_{r-1}}(q^2;q^2)_{s_{r-1}}},\\
\sum_{\lambda \in \tilde{\mathcal{R}}_{i,r}(s_1,\ldots,s_{r-1})} q^{|\lambda|} &= \prod_{j=1}^{\min\{i,r-2\}}\frac{1}{(q)_{s_j-s_{j+1}}} \times\prod_{j=i+1}^{r-2}\frac{q^{(j-i)(s_{j}-s_{j+1})}}{(q)_{s_{j}-s_{j+1}}}\times \frac{q^{(r-i)s_{r-1}}}{(q^2;q^2)_{s_{r-1}}}\\&= \frac{q^{(s_{i+1}+\cdots+s_{r-2}+s_{r-1})+s_{r-1}}}{(q)_{s_1-s_2}\dots(q)_{s_{r-2}-s_{r-1}}(q^2;q^2)_{s_{r-1}}}.
\end{align*}
The proposition follows by summing these identities over all integers $s_1\geq \dots \geq s_{r-1}\geq 0$ and using~\eqref{eq:weightmu}.
\end{proof}

The purpose of the next section is to build a weight- and length-preserving bijection between  
\begin{equation*}
\mathcal{P}_{r-1,r}=\mathcal{P}_{r} \text{  and  }\mathcal{A}_{r-1,r}=\mathcal{A}_{r},
\end{equation*}
therefore proving Theorem~\ref{th:bij}. In Section~\ref{sec:cor}, we will then show that, for all $0\leq i\leq r-1$, this bijection also induces a bijection between
\begin{equation*}\label{subbij}
\mathcal{Q}_{i,r} \text{  and  }\mathcal{T}_{i+1,r},\quad
\mathcal{S}_{i,r} \text{  and  }\mathcal{U}_{i+1,r},\quad
\tilde{\mathcal{S}}_{i,r} \text{  and  }\tilde{\mathcal{U}}_{i+1,r},\quad \mathcal{P}_{i,r}\text{  and  }\mathcal{A}_{i,r},\quad
\mathcal{R}_{i,r} \text{  and  }\mathcal{B}_{i,r},\quad
\tilde{\mathcal{R}}_{i,r} \text{  and  }\tilde{\mathcal{B}}_{i,r}.
\end{equation*}

Then, thanks to Propositions~\ref{prop:prod} and~\ref{prop:multisum}, this will prove Corollaries~\ref{coro:GF}--\ref{coro:list}.

\section{Proof of Theorem~\ref{th:bij}}\label{sec:Bijection}

In this section, we give the bijection between the sets $\mathcal{P}_r$ and $\mathcal{A}_r$. It is in the spirit of Warnaar's bijective proof providing the sum-side of the Andrews--Gordon identities~\cite{W}, which implies~\eqref{eq:AGW}. His idea is to see a certain partition $\mu$ as a \textit{minimal} partition in $\mathcal{T}_{i,r}$, and then insert a $(r-1)$-tuple $(\lambda^{(1)},\ldots,\lambda^{(r-1)})$ of partitions in $\mu$. The process is such that the weight of $\mu$ is incremented after each step, $\mu$ stays in $\mathcal{T}_{i,r}$, and there is a total of $|\lambda^{(1)}|+\cdots+|\lambda^{(r-1)}|$ steps.

For $r$ given non-negative integers  $s_1\geq \cdots \geq s_{r-1}\geq s_r=0$, we consider the minimal partition $\mu(s_1,\ldots,s_{r-1})$ of Definition~\ref{def:mu} which, as noted in the introduction, belongs to $\mathcal{A}_{r}$. We then insert (in a sense that will be defined below) in $\mu(s_1,\ldots,s_{r-1})$ a sequence $(\lambda_0,\ldots,\lambda_{s_1-1})\in \mathcal{P}(s_1,\ldots,s_{r-1})$. Our bijection has a total of $s_1$ steps instead of the $\lambda_0+\cdots+\lambda_{s_1-1}$ steps of Warnaar's, as we insert each part $\lambda_j$ at once, whereas Warnaar was doing it in $\lambda_j$ separate steps.   

In Section \ref{sec:r=2}, we start with a very simple example, namely the case $r=2$. In Sections \ref{sec:Lambda} and \ref{sec:Gamma}, we define maps $\Lambda:\mathcal{P}_{r}\to\mathcal{A}_{r}$ and $\Gamma:\mathcal{A}_{r}\to\mathcal{P}_{r}$, respectively, and show that they are well-defined (see Corollaries~\ref{coro:welldefLambda},~\ref{coro:welldefGamma},~\ref{coro:imLambda}, and~\ref{coro:imGamma}) and weight- and length-preserving (see Corollaries~\ref{coro:Lambdapreserving} and~\ref{coro:Gammapreserving}). Then in Sections \ref{sec:GammaoLambda} and \ref{sec:LambdaoGamma}, we show that
$\Gamma\circ \Lambda$  and $\Lambda\circ\Gamma$ are the identity maps on $\mathcal{P}_{r}$ and $\mathcal{A}_{r}$, respectively (see Propositions~\ref{prop:GrondL} and~\ref{prop:LrondG}). This proves Theorem~\ref{th:bij}.

\subsection{The case $r=2$}\label{sec:r=2}

This case is classical, as the Andrews--Gordon identities for $r=2$ correspond to the famous Rogers--Ramanujan identities. The sum sides of their analytic expressions
$$\sum_{n\geq0}\frac{q^{n^2+n}}{(q)_n}\quad\mbox{and}\quad\sum_{n\geq0}\frac{q^{n^2}}{(q)_n}$$
are usually interpreted as a pair made of a partition and a staircase partition with only even (or only odd) parts. Then it is classical to add the partition to the staircase, obtaining a partition in $\mathcal{T}_{1,2}$ (resp. $\mathcal{T}_{2,2}$) if the staircase partition has even (resp. odd) parts. As we consider partitions that may have parts $0$ with frequency $f_0$, we have to slightly adapt the above method. 

By the definition given in~\eqref{defAr}, the set $\mathcal{A}_{2}$ is made of partitions with frequencies $0$ or $1$, and no pair of consecutive frequencies both equal to $1$. Equivalently, by Proposition~\ref{prop:cdf}~(1), these are partitions whose consecutive parts are at distance at least $2$. For example, the partition $(9,6,4,0)=(1,0,0,0,1,0,1,0,0,1,0,\dots)$ belongs to $\mathcal{A}_{2}$, its length is $4$ and its weight is $19$.

When $r=2$, we only have one integer $s_1=:s$ in Definition~\ref{def:mu}, and the partition $\mu(s)$ is the staircase partition with only even parts from $2s-2$ to $0$. For instance, when $s=4$, we get $\mu(4)=(6,4,2,0)=(1,0,1,0,1,0,1,0,0,\dots)$. By Definition~\ref{def:P}, the set $\mathcal{P}_2$ is made of pairs $(\mu(s),\lambda)$, where $\lambda=(\lambda_0,\dots,\lambda_{s-1})$ is a \textit{non-decreasing} sequence of $s$ integers. For example, the pair $(\mu(4),(0,2,2,3))$ belongs to $\mathcal{P}_{2}$, its length is $4$ (the length of $\mu(4)$) and its weight is $19$.

Our map $\Lambda:\mathcal{P}_{2}\to\mathcal{A}_{2}$ starts with an element $(\mu(s),\lambda)\in\mathcal{P}_{2}$, and adds the integer $\lambda_{s-1}$ to the first part $2s-2$ of the staircase, then the integer $\lambda_{s-2}$ to the second part $2s-4$ of the staircase, and so on until adding the integer $\lambda_0$ to the last part $0$ of the staircase. The resulting partition is therefore $(\lambda_{s-1}+2s-2,\lambda_{s-2}+2s-4,\dots,\lambda_0+0)$ which belongs to $\mathcal{A}_{2}$, has length $s$ and weight $|\lambda|+s^2-s$. 

An example is depicted in Figure~\ref{fig:fig1r=2} for $(\mu(4),(0,2,2,3))$.

\begin{figure}[H]
\includegraphics[width=0.7\textwidth]{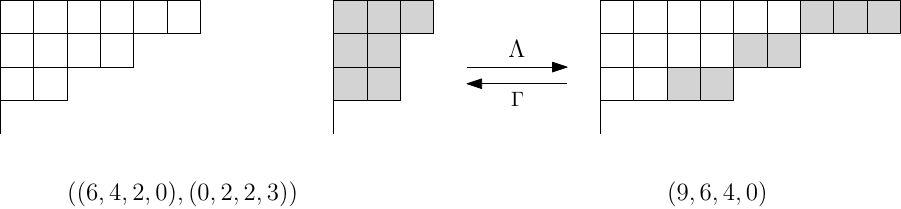}
\caption{The maps in terms of parts.}
\label{fig:fig1r=2}
\end{figure}

To generalize this to $\mathcal{P}_{r}$, we need to describe this map, easily explained in terms of parts, in terms of frequencies. We first have to identify the greatest index with non-zero frequency in $\mu(s)=(f_j)_{j\geq0}$, namely $2s-2$, and shift $f_{2s-2}$ from $1$ to $0$ while $f_{2s-2+\lambda_{s-1}}$ is shifted from $0$ to $1$. We successively do the same shifts for $f_{2s-4}$ using $\lambda_{s-2}$, \dots, $f_0$ using $\lambda_0$.
Figure~\ref{fig:fig2r=2} represents the same map for $(\mu(4),(0,2,2,3))$ as in Figure~\ref{fig:fig1r=2}, but in terms of frequencies, and with notation from Section \ref{sec:Lambda}.

\begin{figure}[H]
\includegraphics[width=0.7\textwidth]{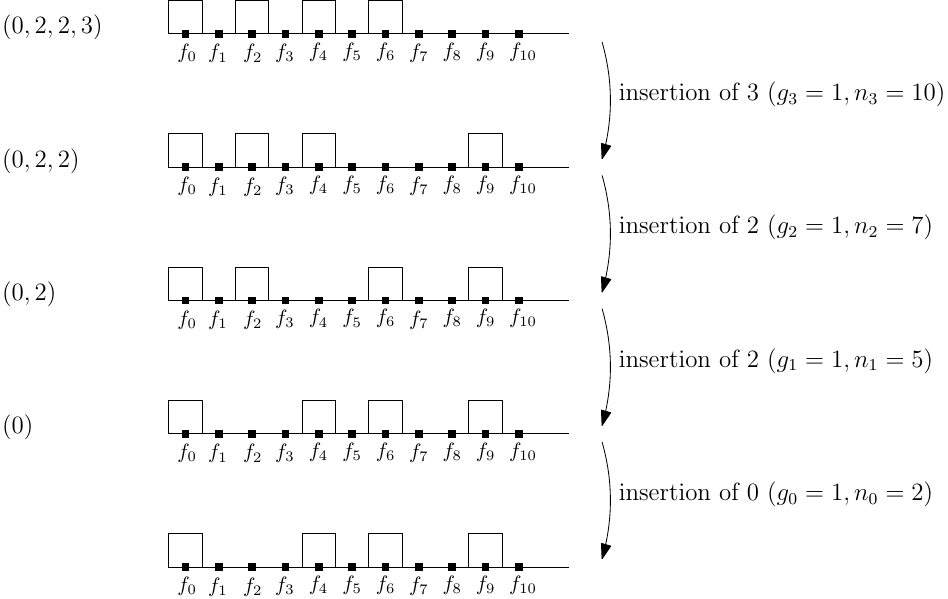}
\caption{The map $\Lambda$ in terms of frequencies.}
\label{fig:fig2r=2}
\end{figure}


Our map $\Gamma:\mathcal{A}_{2}\to\mathcal{P}_{2}$ starts with a partition $\nu=(\nu_1,\dots,\nu_s)\in\mathcal{A}_{2}$, and extracts from $\nu_s$ the part $0$, then from $\nu_{s-1}$ the part $2$, and so on until extraction from $\nu_1$ of the part $2s-2$. The result is a pair made of $\mu(s)$ and a non-decreasing sequence $(\nu_s-0,\nu_{s-1}-2,\dots,\nu_1-(2s-2))$  of length $s$: this pair therefore belongs to $\mathcal{P}_{2}$ and has weight $|\nu|$. See Figure~\ref{fig:fig1r=2} for $\Gamma(9,6,4,0)$.

Again, to generalize this, we need to describe this process in terms of frequencies: we first have to identify the smallest index with non-zero frequency in $\nu=(f_j)_{j\geq0}$, namely $\nu_s$, and shift $f_{\nu_s}$ from $1$ to $0$ while $f_{0}$ is shifted from $0$ to $1$ (or kept unchanged if $\nu_s=0$) and we keep track of the extracted $\nu_s$. We successively do the same shifts for $f_{\nu_{s-1}}$ and $f_2$ (keeping track of the extracted $\nu_{s-1}-2$), \dots, $f_{\nu_1}$ and $f_{2s-2}$ (keeping track of the extracted $\nu_{1}-(2s-2)$).
Figure~\ref{fig:fig3r=2} represents the same map for $(9,6,4,0)$ as in Figure~\ref{fig:fig1r=2}, but in terms of frequencies and with notation from the upcoming Section \ref{sec:Gamma}.
\begin{figure}[H]
\includegraphics[width=0.7\textwidth]{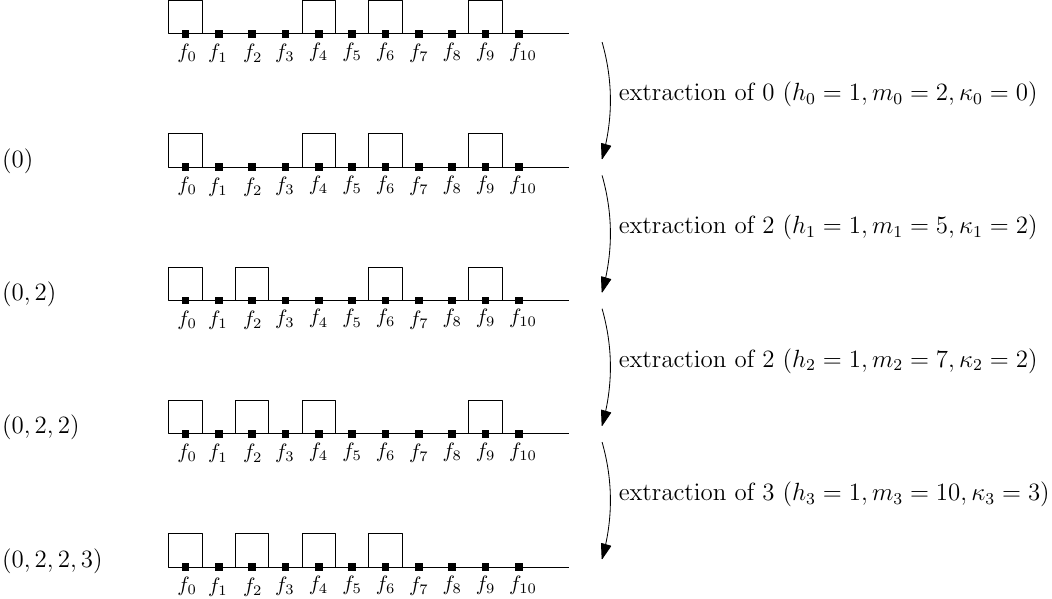}
\caption{The map $\Gamma$ in terms of frequencies.}
\label{fig:fig3r=2}
\end{figure}

\subsection{The map $\Lambda:\mathcal{P}_{r}\to\mathcal{A}_{r}$}\label{sec:Lambda}

Let $s_1\geq \cdots \geq s_{r-1}\geq s_r=0$ be integers, set $s_0=\infty$.

For every non-negative integer $u$, define $g_u $ to be the unique integer in $\{0,\ldots,r-1\}$ such that $s_{g_u+1}\leq u < s_{g_u}$. In other words, $g_u$ is the largest $j$ such that $s_j$ is bigger than $u$. For instance $g_{s_1}=0$, and we have by convention $g_{s_0}=0$. On the example given in Figure~\ref{fig:gu}, we have $g_u=4$, as $s_{5}\leq u < s_{4}$.
\begin{figure}[H]
\includegraphics[width=0.8\textwidth]{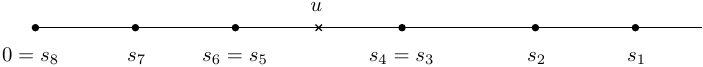}
\caption{An example when $r=8$.}
\label{fig:gu}
\end{figure}

Let $\lambda=(\lambda_0,\ldots,\lambda_{s_1-1}) \in \mathcal{P}(s_1,\ldots,s_{r-1})$. 
The principle of our bijection $\Lambda$ is to insert the parts of $\lambda$ one by one in $\mu(s_1, \ldots, s_{r-1})$, while preserving the length of $\mu(s_1, \ldots, s_{r-1})$ (see the detailed properties of the bijection in Propositions~\ref{prop:welldefnu},~\ref{prop:welldeflambda_s1}, and~\ref{prop:welldeflambda_s1-1}).

Let $\theta^{(s_1)}=\left(\theta_j^{(s_1)}\right)_{j\geq 0}$ be the multiplicity sequence of $\mu(s_1, \ldots, s_{r-1})$, and construct the sequences $\theta^{(u)}=\left(\theta_j^{(u)}\right)_{j\geq 0}$ recursively in decreasing order according to $u \in \{0,\ldots,s_1-1\}$.

Recall from Definition \ref{def:mu} that for all $0\leq j\leq r-1$,
\begin{equation}\label{eq:rappel}
(\theta_{2k}^{(s_1)},\theta_{2k+1}^{(s_1)})=(j,0)\quad\text{ for all }\quad s_{j+1}\leq k<s_j.
\end{equation}

Suppose that the sequence $\theta^{(u+1)}$ is built. Let
\begin{equation}\label{eq:defru}
n_u := \min\left\{t\geq 2u+2: \sum_{j=2u+2}^{t} \left[g_u-\left(\theta_j^{(u+1)}+\theta_{j-1}^{(u+1)}\right)\right] \geq \lambda_u\right\}.
\end{equation}
We will prove in Proposition~\ref{prop:welldefnu} that $n_u$ is well-defined for all $u \in \{0,\ldots,s_1-1\}$.

Now construct the sequence $\theta^{(u)}$ by modifying $\theta^{(u+1)}$ as follows: 
\begin{align}
\theta_j^{(u)}&:=\theta_j^{(u+1)}\text{ if } 0\leq j<2u \text{ or } j\geq n_u&\text{(fixed)},\label{eq:fixed1}\\
\theta_j^{(u)}&:=\theta_{j+2}^{(u+1)}\text{ if } 2u\leq j<n_u-2&\text{(shifted twice to the left)},\label{eq:decalage1}\\
\theta_{n_u-1}^{(u)} &= g_u- \theta_{n_u-2}^{(u)} := \theta_{n_u-1}^{(u+1)}+\lambda_u\nonumber\\
&\hskip 3cm-\sum_{j=2u+2}^{n_u-1} \left[g_u-\left(\theta_j^{(u+1)}+\theta_{j-1}^{(u+1)}\right)\right]&\text{(modified and moved to the right)}\label{eq:def1}.
\end{align}

Figure~\ref{fig:decalage1} gives an illustration of how the multiplicities are modified from step $u+1$ to step $u$.

\begin{figure}[H]
\begin{center}
\begin{tikzpicture}[scale=0.9, every node/.style={scale=0.7}]

\draw [fill=blue] (-4-4,0.5)--(-4-4,0.75)--(-3.75-4,0.75)--(-3.75-4,0.5)--cycle;

\draw [fill=foge] (-4-4,-0.5)--(-4-4,-0.25)--(-3.75-4,-0.25)--(-3.75-4,-0.5)--cycle;

\draw [fill=red!80] (-4-4,-1.5)--(-4-4,-1.25)--(-3.75-4,-1.25)--(-3.75-4,-1.5)--cycle;

\draw (-7.6,0.6) node [right] {fixed when not in $\{2u,\ldots,n_u-1\}$};

\draw (-7.6,-0.4) node [right] {shifted twice to the left when in $\{2u+2,\ldots,n_u-1\}$};

\draw (-7.6,-1.4) node [right] {modified and moved to the right};

\draw (-7.6,-1.8) node [right] {from $(2u,2u+1)$ to $(n_u-2,n_u-1)$};

\draw (6,0.5) node[right] {\begin{Large}Step $u+1$ \end{Large}};
\draw (6,-1.75) node[right] {\begin{Large}Step $u$ \end{Large}};

\draw [thick,->] (0.25,0)--(5,0);

\foreach \x in {0.25,0.5,...,4.5}
\draw (0.125
+\x,0.05)--(0.125+\x,-0.05);

\draw [thick,->] (0.25,-2)--(5,-2);

\foreach \x in {0.25,0.5,...,4.5}
\draw (0.125
+\x,-1.95)--(0.125+\x,-2.05);


\draw [fill=red!80] (0.75,0)--(0.75,0.75)--(1,0.75)--(1,0)--(1.25,0)--cycle;


\draw [fill=red!80] (3,0-2)--(3,0.25-2)--(3.25,0.25-2)--(3.25,0-2)--(3.25,0.5-2)--(3.5,0.5-2)--(3.5,0-2)--cycle;


\foreach \x in {0,1}
\draw [fill=foge, dashed] (1.25-0.5*\x,0-2*\x)--(1.25-0.5*\x,0.5-2*\x)--(1.5-0.5*\x,0.5-2*\x)--(1.5-0.5*\x,0.25-2*\x)--(1.75-0.5*\x,0.25-2*\x)--(1.75-0.5*\x,0.75-2*\x)--(2-0.5*\x,0.75-2*\x)--(2-0.5*\x,0.5-2*\x)--(2.25-0.5*\x,0.5-2*\x)--(2.25-0.5*\x,1-2*\x)--(2.5-0.5*\x,1-2*\x)--(2.5-0.5*\x,0.5-2*\x)--(3-0.5*\x,0.5-2*\x)--(3-0.5*\x,0-2*\x)--(3.25-0.5*\x,0-2*\x)--(3.25-0.5*\x,1.25-2*\x)--(3.5-0.5*\x,1.25-2*\x)--(3.5-0.5*\x,0-2*\x);


\foreach \x in {0,1}
\draw [fill=blue!80, dashed] (3.5,0-2*\x)--(3.5,0.25-2*\x)--(3.75,0.25-2*\x)--(3.75,1.25-2*\x)--(4,1.25-2*\x)--(4,0.5-2*\x)--(4.25,0.5-2*\x)--(4.25,0.75-2*\x)--(4.75,0.75-2*\x)--(4.75,0-2*\x);


\foreach \x in {0,1}
\draw [fill=blue!80, dashed] (0.25,0-2*\x)--(0.25,0.75-2*\x)--(0.5,0.75-2*\x)--(0.5,0-2*\x);

\end{tikzpicture}
\end{center}
\caption{Effects of $\Lambda$ on the multiplicity sequence from step $u+1$ to step $u$.}
\label{fig:decalage1}
\end{figure}
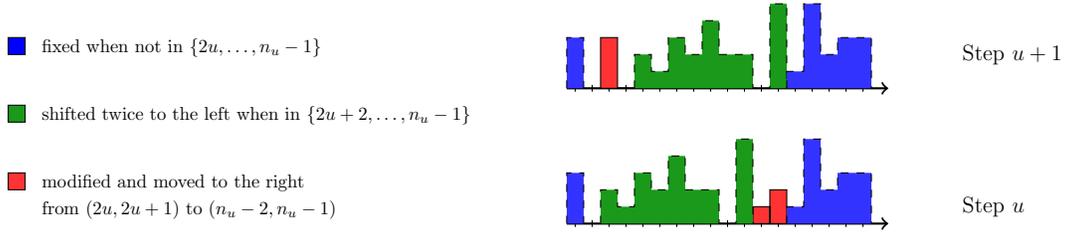

Finally, define $\nu$ to be the partition with frequency sequence $(\theta^{(0)}_j)_{j\geq0}$ (we show in Proposition \ref{prop:welldeflambda_s1} (2) that these frequencies are indeed non-negative), and set
$$\Lambda(\mu(s_1,\ldots,s_{r-1}),\lambda):=\nu.$$

\begin{ex} For $r=4$, $(s_1,s_2,s_3)=(4,2,2)$ and $(\lambda_0,\lambda_1,\lambda_2,\lambda_3)=(5,6,1,3)$, we have
$$\theta^{(4)} = (3,0,3,0,1,0,1,0,0,\ldots).$$
\begin{itemize}
\item At step $3$, we obtain $g_3=1$, and $n_3=10$,
$$\theta^{(3)} = (3,0,3,0,1,0,\textcolor{foge}{0,0},\textcolor{red!80}{0,1},0,\ldots).$$
\item At step $2$, we obtain $g_2=1$, and $n_2=6$,
$$\theta^{(2)} = (3,0,3,0,\textcolor{red!80}{0,1},0,0,0,1,0,\ldots).$$
\item At step $1$, we obtain $g_1=3$, and $n_1=6$,
$$\theta^{(1)} = (3,0,\textcolor{foge}{0,1},\textcolor{red!80}{1,2},0,0,0,1,0,\ldots).$$
\item At step $0$, we obtain $g_0=3$, and $n_0=3$,
$$\theta^{(0)} = (\textcolor{foge}{0},\textcolor{red!80}{1,2},1,1,2,0,0,0,1,0,\ldots).$$
\end{itemize} 
Hence, $\Lambda(\mu(4,2,2),(5,6,1,3))=\nu$ with multiplicity sequence $(0,1,2,1,1,2,0,0,0,1,0,\ldots)$. 
\end{ex}

The following property proves that $\Lambda$ is well-defined.
\begin{Proposition}\label{prop:welldefnu}
For all $u \in \{0,\ldots,s_1-1\}$, $\theta^{(u+1)}_{j}=0$ for $j$ large enough, and $n_u$ is well-defined.
\end{Proposition}
\begin{proof}
This follows from a simple backward induction on $u$. By the definition of $\theta^{(s_1)}$, $\theta^{(s_1)}_{j}=0$ for $j \geq 2 s_1$.

Now assume the proposition is true for $u+1$ and show it for $u$. As $\theta^{(u+1)}_{j}=0$ for $j$ large enough, there is some integer $t\geq 2u+2$ such that $\sum_{j=2u+2}^{t} \left[g_u-\left(\theta_j^{(u+1)}+\theta_{j-1}^{(u+1)}\right)\right] \geq \lambda_u$ (indeed for all $u \in \{0,\ldots,s_{1}-1\}$, $g_u>0$), and $n_u$ is well-defined. Finally, using \eqref{eq:fixed1}, we deduce that $\theta^{(u)}_{j}=0$ for $j$ large enough.
\end{proof}

\begin{Corollary}\label{coro:welldefLambda}
The map $\Lambda$ is well-defined.
\end{Corollary}

Now to show that the image of $\Lambda$ is in $\mathcal{A}_{r}$, we first need some additional key properties. 

\begin{Proposition}\label{prop:welldeflambda_s1}
For all $u \in \{0,\ldots,s_{1}\}$, the following holds:
\begin{enumerate}
\item For all $1\leq j<2u$, $\theta^{(u)}_{j} = \theta^{(s_1)}_{j}$.
\item For all $j\geq 0$, $\theta^{(u)}_{j}\geq 0$.
\item For all $j\geq 2u$, $g_u\geq \theta^{(u)}_{j}+\theta^{(u)}_{j+1}.$ 
\end{enumerate}
\end{Proposition}

\begin{proof}
\begin{enumerate}
\item This is clear by backward induction on $u$, using \eqref{eq:fixed1}.

\item This is again proved by backward induction on $u$. For $u=s_1$, by definition $\theta^{(s_1)}_{j} \geq 0$ for all $j \geq 0$.
Now assume that $\theta^{(u+1)}_{j}\geq 0$ for all $j \geq 0$ and show that $\theta^{(u)}_{j}\geq 0$ for all $j \geq 0$.
\begin{itemize}
\item If $j\notin \{n_u-1, n_u-2\}$, then it is immediate by \eqref{eq:fixed1} and \eqref{eq:decalage1}.
\item If $j=n_u-1$, then by \eqref{eq:def1},
$$\theta_{n_u-1}^{(u)} = \theta_{n_u-1}^{(u+1)}+\lambda_u-\sum_{j=2u+2}^{n_u-1} \left[g_u-\left(\theta_j^{(u+1)}+\theta_{j-1}^{(u+1)}\right)\right].$$
By definition of $n_u$, the last sum is at most $\lambda_u$, with equality only if $\lambda_u=0$ and $n_u=2u+2$. Hence
\begin{equation}
\label{eq:petit2}
\theta_{n_u-1}^{(u)} \geq \theta_{n_u-1}^{(u+1)} \geq 0.
\end{equation}
\item If $j=n_u-2$, then by \eqref{eq:def1},
\begin{align*}
\theta_{n_u-2}^{(u)} &= g_u -\theta_{n_u-1}^{(u+1)} -\lambda_u +\sum_{j=2u+2}^{n_u-1} \left[g_u-\left(\theta_j^{(u+1)}+\theta_{j-1}^{(u+1)}\right)\right]\\
&= \theta_{n_u}^{(u+1)} -\lambda_u +\sum_{j=2u+2}^{n_u} \left[g_u-\left(\theta_j^{(u+1)}+\theta_{j-1}^{(u+1)}\right)\right].
\end{align*}
By definition of $n_u$, the last sum is at least $\lambda_u$. Hence
\begin{equation}
\label{eq:petit1}
\theta_{n_u-2}^{(u)} \geq \theta_{n_u}^{(u+1)} \geq 0.
\end{equation}
\end{itemize}

\item Let us do a last backward induction on $u$. For $u=s_1$, by the definition of $\theta^{(s_1)}$, $\theta^{(s_1)}_{j}=0$ for $j \geq 2 s_1$. Hence $g_{s_1}=0 \geq  \theta^{(s_1)}_{j}+\theta^{(s_1)}_{j+1}.$

Now assume that $g_{u+1}\geq \theta^{(u+1)}_{j}+\theta^{(u+1)}_{j+1}$ for all $j\geq 2u+2$, and show that $g_u\geq \theta^{(u)}_{j}+\theta^{(u)}_{j+1}$ for all $j\geq 2u$.
\begin{itemize}
\item If $j\geq n_u\geq 2u+2$, then by \eqref{eq:fixed1},
$$\theta^{(u)}_{j}+\theta^{(u)}_{j+1} =\theta^{(u+1)}_{j}+\theta^{(u+1)}_{j+1} \leq g_{u+1} \leq g_u.$$ 
\item If $2u\leq j<n_u-3$, then by~\eqref{eq:decalage1},
$$\theta^{(u)}_{j}+\theta^{(u)}_{j+1} = \theta^{(u+1)}_{j+2}+\theta^{(u+1)}_{j+3} \leq g_{u+1} \leq g_u.$$ 
\item If $j=n_u-1$, then by \eqref{eq:petit1} and \eqref{eq:fixed1},
$$\theta^{(u)}_{n_u-1}+\theta^{(u)}_{n_u} = g_u - \theta_{n_u-2}^{(u)} +\theta^{(u)}_{n_u} \leq g_u -\theta^{(u+1)}_{n_u} +\theta^{(u+1)}_{n_u} \leq g_u.$$ 
\item If $j=n_u-2$, then by \eqref{eq:def1},
$$\theta^{(u)}_{n_u-2}+\theta^{(u)}_{n_u-1} = g_u.$$ 
\item If $j=n_u-3$, then by~\eqref{eq:decalage1} and \eqref{eq:petit2},
$$\theta^{(u)}_{n_u-3}+\theta^{(u)}_{n_u-2} = \theta_{n_u-1}^{(u+1)} +\theta^{(u)}_{n_u-2} \leq \theta_{n_u-1}^{(u)} +\theta^{(u)}_{n_u-2} = g_u.$$ 
\end{itemize}
\end{enumerate}
\end{proof}

\begin{Corollary}\label{coro:imLambda}
The image of $\Lambda$ is in $\mathcal{A}_{r}$.
\end{Corollary}
\begin{proof}
By Proposition~\ref{prop:welldeflambda_s1}~(2) with $u=0$, the integers $\theta^{(0)}_j$ are non-negative for $j\geq0$, so $\Lambda(\mu(s_1,\ldots,s_{r-1}),\lambda)$ is a partition.
By definition of $g_u$, we know that $0\leq g_0\leq r-1$. Thus by Proposition~\ref{prop:welldeflambda_s1}~(3) with $u=0$, we have $\theta^{(0)}_{j}+\theta^{(0)}_{j+1} \leq r-1$ for all $j\geq 0$. Hence for all $s_1\geq \cdots \geq s_{r-1}$ and $\lambda \in \mathcal{P}(s_1,\ldots,s_{r-1})$, $\Lambda(\mu(s_1,\ldots,s_{r-1}),\lambda)$ belongs to $\mathcal{A}_{r}$.
\end{proof}

\medskip

Finally, we state a few additional properties to show that $\Lambda$ is weight- and length-preserving. 
\begin{Proposition}\label{prop:welldeflambda_s1-1}
For all $u \in \{0,\ldots,s_{1}-1\}$, the following holds:
\begin{enumerate}
\item $(\theta^{(u+1)}_{2u},\theta^{(u+1)}_{2u+1})=(g_u,0)$,
\item the length of $\theta^{(u)}$ equals the length of $\theta^{(u+1)}$, i.e. $\sum_{j\geq 0} \theta^{(u)}_{j} = \sum_{j\geq 0} \theta^{(u+1)}_{j}$,
\item the weight of $\theta^{(u)}$ is $\lambda_u$ more than the weight of $\theta^{(u+1)}$, i.e. $\sum_{j\geq 0} j\cdot \theta^{(u)}_{j} = \lambda_u+\sum_{j\geq 0} j\cdot\theta^{(u+1)}_{j}$.
\end{enumerate}
\end{Proposition}

\begin{proof}
\begin{enumerate}
\item By Proposition \ref{prop:welldeflambda_s1} (1), for all $1\leq j<2u+2$, $\theta^{(u+1)}_{j} = \theta^{(s_1)}_{j}$. In particular, by \eqref{eq:rappel},
$$(\theta^{(u+1)}_{2u},\theta^{(u+1)}_{2u+1})=(\theta^{(s_1)}_{2u},\theta^{(s_1)}_{2u+1})=(g_u,0).$$

\item We have 
\begin{align*}
\sum_{j\geq 0} \theta_{j}^{(u)} &=  \sum_{j=0}^{2u-1} \theta_{j}^{(u)} + \sum_{j= 2u}^{n_u-3} \theta_{j}^{(u)} + g_u +  \sum_{j\geq n_u} \theta_{j}^{(u)}&\text{(by ~\eqref{eq:def1})}\\
&=  \sum_{j=0}^{2u-1} \theta_{j}^{(u+1)} + \sum_{j= 2u+2}^{n_u-1} \theta_{j}^{(u+1)} + g_u +  \sum_{j\geq n_u} \theta_{j}^{(u+1)}&\text{(by ~\eqref{eq:fixed1} and ~\eqref{eq:decalage1})}\\
&=\sum_{j\geq 0} \theta_{j}^{(u+1)}&\text{(by Proposition ~\ref{prop:welldeflambda_s1-1} (1))}.
\end{align*}

\item  We have
\begin{align*}
\sum_{j\geq 0} j\cdot \theta_{j}^{(u)} &=  \sum_{j=0}^{2u-1} j\cdot \theta_{j}^{(u)} + \sum_{j= 2u}^{n_u-3} j\cdot\theta_{j}^{(u)} + g_u\cdot (n_u-2)+ \theta_{n_u-1}^{(u)}+  \sum_{j\geq n_u} j\cdot\theta_{j}^{(u)} &\text{(by ~\eqref{eq:def1})}\\
&=  \sum_{j=0}^{2u-1} j\cdot \theta_{j}^{(u+1)} + \sum_{j= 2u+2}^{n_u-1} (j-2)\cdot\theta_{j}^{(u+1)} +  g_u\cdot(n_u-2)+ \theta_{n_u-1}^{(u)}\\
&\hskip 8cm +  \sum_{j\geq n_u} j\cdot \theta_{j}^{(u+1)}&\text{(by ~\eqref{eq:fixed1} and ~\eqref{eq:decalage1})}\\
&=  \sum_{j=0}^{2u-1} j\cdot \theta_{j}^{(u+1)} + \sum_{j= 2u+2}^{n_u-1} j\cdot \theta_{j}^{(u+1)} +  2u\cdot g_u+ \lambda_u+ \theta_{2u+1}^{(u+1)}\\
&\hskip 8cm +  \sum_{j\geq n_u} j\cdot \theta_{j}^{(u+1)}&\text{(by ~\eqref{eq:def1})}\\
&= \lambda_u +\sum_{j\geq 0} j\cdot \theta_{j}^{(u+1)} &\text{(by Proposition ~\ref{prop:welldeflambda_s1-1} (1))}.
\end{align*}
\end{enumerate}
\end{proof}

\begin{Corollary}\label{coro:Lambdapreserving}
$\Lambda$ preserves the weight and the length, i.e., for all $s_1\geq \cdots \geq s_{r-1}$ and $\lambda \in \mathcal{P}(s_1,\ldots,s_{r-1})$,
$$|\Lambda(\mu(s_1,\ldots,s_{r-1}),\lambda)|=|\lambda|+|\mu(s_1,\ldots,s_{r-1})|=|(\mu(s_1,\ldots,s_{r-1}),\lambda)|,$$
and the length of $\Lambda(\mu(s_1,\ldots,s_{r-1}),\lambda)$ is equal to the length of $\mu(s_1,\ldots,s_{r-1})$.
\end{Corollary}
\begin{proof}
This is a direct consequence of Proposition~\ref{prop:welldeflambda_s1-1}~(2) and (3).
\end{proof}

\subsection{The map $\Gamma:\mathcal{A}_{r}\to\mathcal{P}_{r}$}\label{sec:Gamma}

Let $\nu\in \mathcal{A}_{r}$, and let $\eta^{(0)}=\left(\eta_j^{(0)}\right)_{j\geq 0}$ be its multiplicity sequence. The idea here is to retrieve from $\nu$ a pair $(\mu,\kappa)$ in $\mathcal{P}_r$. We construct the sequences $\eta^{(u)}=\left(\eta_j^{(u)}\right)_{j\geq 0}$ recursively in increasing order for $u$ as follows. Suppose that the sequence $\eta^{(u)}$ has been constructed.

Define $h_u := \max\left\{\eta_j^{(u)}+\eta^{(u)}_{j+1}: j\geq 2u\right\}$, and
\begin{equation}\label{eq:defsu}
m_u:= \min\left\{j\geq 2u+2: \eta_{j-2}^{(u)}+\eta^{(u)}_{j-1}=h_u\right\}.
\end{equation}
We will prove in Proposition~\ref{prop:welldefmu} that $h_u$ and $m_u$ are well-defined for all $u \geq 0$.

Now construct $\eta^{(u+1)}$ as follows:
\begin{align}
\eta_j^{(u+1)}&=\eta_j^{(u)}\text{ if } 0\leq j<2u \text{ or } j\geq m_u &\text{(fixed)},\label{eq:fixed2}\\
\eta_j^{(u+1)}&=\eta_{j-2}^{(u)}\text{ if } 2u+2\leq j<m_u &\text{(shifted twice to the right)}, \label{eq:decalage2}\\
(\eta_{2u}^{(u+1)},\eta_{2u+1}^{(u+1)})&= (h_u,0) &\text{(modified and moved to the left)}.\label{eq:def2}
\end{align}

Moreover we define
\begin{equation}
\kappa_u :=  h_u-\eta_{2u}^{(u)}+ \sum_{j=2u}^{m_u-3}\left[h_u-\left(\eta_j^{(u)}+\eta_{j+1}^{(u)}\right)\right]. \label{eq:defkappa}
\end{equation}

Figure~\ref{fig:decalage2} gives an illustration of how the multiplicities are modified from step $u$ to step $u+1$.

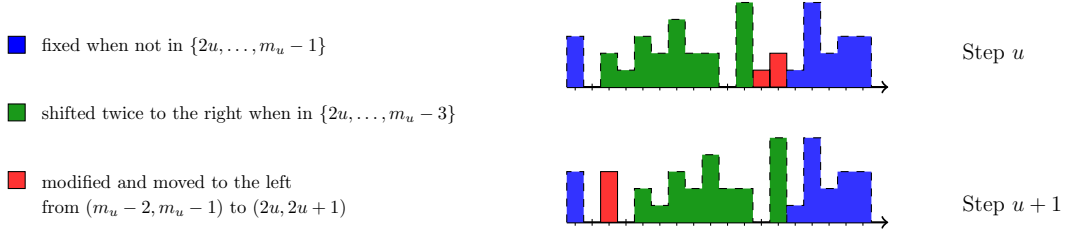
\begin{figure}[H]
\begin{center}
\begin{tikzpicture}[scale=0.9, every node/.style={scale=0.7}]

\draw [fill=blue] (-4-4,0.5)--(-4-4,0.75)--(-3.75-4,0.75)--(-3.75-4,0.5)--cycle;

\draw [fill=foge] (-4-4,-0.5)--(-4-4,-0.25)--(-3.75-4,-0.25)--(-3.75-4,-0.5)--cycle;

\draw [fill=red!80] (-4-4,-1.5)--(-4-4,-1.25)--(-3.75-4,-1.25)--(-3.75-4,-1.5)--cycle;

\draw (-7.6,0.6) node [right] {fixed when not in $\{2u,\ldots,m_u-1\}$};

\draw (-7.6,-0.4) node [right] {shifted twice to the right when in $\{2u,\ldots,m_u-3\}$};

\draw (-7.6,-1.4) node [right] {modified and moved to the left};

\draw (-7.6,-1.8) node [right] {from $(m_u-2,m_u-1)$ to $(2u,2u+1)$};

\draw (6,0.5) node[right] {\begin{Large}Step $u$ \end{Large}};
\draw (6,-1.75) node[right] {\begin{Large}Step $u+1$ \end{Large}};

\draw [thick,->] (0.25,0)--(5,0);

\foreach \x in {0.25,0.5,...,4.5}
\draw (0.125
+\x,0.05)--(0.125+\x,-0.05);

\draw [thick,->] (0.25,-2)--(5,-2);

\foreach \x in {0.25,0.5,...,4.5}
\draw (0.125
+\x,-1.95)--(0.125+\x,-2.05);


\draw [fill=red!80] (0.75,0-2)--(0.75,0.75-2)--(1,0.75-2)--(1,0-2)--cycle;


\draw [fill=red!80] (3,0)--(3,0.25)--(3.25,0.25)--(3.25,0)--(3.25,0.5)--(3.5,0.5)--(3.5,0)--cycle;


\foreach \x in {0,1}
\draw [fill=foge, dashed] (1.25-0.5*\x,0-2+2*\x)--(1.25-0.5*\x,0.5-2+2*\x)--(1.5-0.5*\x,0.5-2+2*\x)--(1.5-0.5*\x,0.25-2+2*\x)--(1.75-0.5*\x,0.25-2+2*\x)--(1.75-0.5*\x,0.75-2+2*\x)--(2-0.5*\x,0.75-2+2*\x)--(2-0.5*\x,0.5-2+2*\x)--(2.25-0.5*\x,0.5-2+2*\x)--(2.25-0.5*\x,1-2+2*\x)--(2.5-0.5*\x,1-2+2*\x)--(2.5-0.5*\x,0.5-2+2*\x)--(3-0.5*\x,0.5-2+2*\x)--(3-0.5*\x,0-2+2*\x)--(3.25-0.5*\x,0-2+2*\x)--(3.25-0.5*\x,1.25-2+2*\x)--(3.5-0.5*\x,1.25-2+2*\x)--(3.5-0.5*\x,0-2+2*\x);


\foreach \x in {0,1}
\draw [fill=blue!80, dashed] (3.5,0-2*\x)--(3.5,0.25-2*\x)--(3.75,0.25-2*\x)--(3.75,1.25-2*\x)--(4,1.25-2*\x)--(4,0.5-2*\x)--(4.25,0.5-2*\x)--(4.25,0.75-2*\x)--(4.75,0.75-2*\x)--(4.75,0-2*\x);


\foreach \x in {0,1}
\draw [fill=blue!80, dashed] (0.25,0-2*\x)--(0.25,0.75-2*\x)--(0.5,0.75-2*\x)--(0.5,0-2*\x);

\end{tikzpicture}
\end{center}
\caption{Effects of $\Gamma$ on the multiplicity sequence from step $u$ to step $u+1$.}
\label{fig:decalage2}
\end{figure}

\begin{Remark}\label{rem:eta_positive}
For all non-negative integers $u,j$, we have $\eta^{(u)}_j \geq 0$. This follows directly from the fact that $\eta^{(0)}$ is the frequency sequence of a partition, and inductively from equations \eqref{eq:fixed2}--\eqref{eq:def2}.
\end{Remark}

\begin{Remark}
Note that if $h_u=0$ for some $u$, then by Remark~\ref{rem:eta_positive} and~\eqref{eq:defsu}--\eqref{eq:def2}, $\eta^{(v)}=\eta^{(u)}$ for all $v \geq u$. We will show in Corollary \ref{coro:welldefGamma} that such a $u$ always exist.
\end{Remark}

Let $U$ be the smallest $u$ such that $h_u=0$. We stop the recursive process of building $\eta^{(u)}$ at $u=U$, and define the image of $\nu$ by $\Gamma$ as follows.
Let $\mu$ be the partition with multiplicity sequence $\left(\eta_j^{(U)}\right)_{j\geq 0}$, let $\kappa$ be the sequence $(\kappa_0,\ldots,\kappa_{U-1})$ (or the empty sequence if $U=0$), and set
$$\Gamma(\nu):=(\mu,\kappa).$$

\begin{ex} For $r=4$, let $\nu$ be the partition with multiplicity sequence
$$\eta^{(0)} = (0,1,2,1,1,2,0,0,0,1,0,\ldots).$$
\begin{itemize}
\item At step $0$, we obtain $h_0=3$, $m_0=3$,
$$\eta^{(1)} = (\textcolor{red!80}{3,0},\textcolor{foge}{0},1,1,2,0,0,0,1,0,\ldots),$$
and $\kappa_0=5$.
\item At step $1$, we obtain $h_1=3$, $m_1=6$,
$$\eta^{(2)} = (3,0,\textcolor{red!80}{3,0},\textcolor{foge}{0,1},0,0,0,1,0,\ldots),$$
and $\kappa_1=6$.
\item At step $2$, we obtain $h_2=1$, $m_2=6$,
$$\eta^{(3)} = (3,0,3,0,\textcolor{red!80}{1,0},0,0,0,1,0,\ldots),$$
and $\kappa_2=1$.
\item At step $3$, we obtain $h_3=1$, $m_3=10$,
$$\eta^{(4)} = (3,0,3,0,1,0,\textcolor{red!80}{1,0},\textcolor{foge}{0,0},0,\ldots),$$
and $\kappa_3=3$.
\item At step $4$, we obtain $h_4=0$, so we stop the process.
\end{itemize} 
Therefore, $\Gamma(\nu) = (\mu(4,2,2),(5,6,1,3))$. 
\end{ex}

First check that $\Gamma$ is well-defined, using the following propositions. 

\begin{Proposition}\label{prop:welldefmu}
Let $L$ be the largest part of the partition $\nu$. Then for all non-negative integers $u$, the quantities $h_u$ and $m_u$ are well-defined, and for all $j \geq L +1$, $\eta^{(u)}_{j}=0$.
\end{Proposition}

\begin{proof}
This follows by induction on $u$. As $(\eta^{(0)}_{j})_{j\geq0}$ is the multiplicity sequence of the partition $\nu$, by definition of $L$, for all $j \geq L +1$, $\eta^{(0)}_{j}=0$. Hence $h_0$ and $m_0$ are well-defined.

Now assume the proposition is true for $u\geq0$ and show it for $u+1$. We distinguish two cases:
\begin{itemize}
\item If $m_u=2u+2=L+2$, then by \eqref{eq:def2}, $\eta_{2u+1}^{(u+1)}=\eta_{L+1}^{(u+1)}=0$, and by \eqref{eq:fixed2}, for all $j \geq L +2$, $\eta_{j}^{(u+1)}=0$.
\item Otherwise, by \eqref{eq:defsu}, $m_u \in \{2u+2, \dots , L +1\}$. Then by \eqref{eq:fixed2}, for all $j \geq L +1$, $\eta_{j}^{(u+1)}=0$.
\end{itemize}
Hence for all $j \geq L +1$, $\eta^{(u+1)}_{j}=0$. Thus $h_{u+1}$, and therefore $m_{u+1}$, is well-defined.
\end{proof}

\begin{Corollary}\label{coro:welldefGamma}
The map $\Gamma$ is well-defined. In particular, $h_u=0$ for $u$ large enough and $U$ is well-defined.
\end{Corollary}

\begin{proof}
Thanks to Proposition~\ref{prop:welldefmu}, $h_u$ and $m_u$ are well-defined. It remains to show that there exists $u$ such that $h_u=0$, so that $U$ is well-defined and the process stops. From Proposition~\ref{prop:welldefmu}, for all $j \geq L +1$, $\eta^{(u)}_{j}=0$. Thus for all $u \geq (L +1)/2$, $h_u=0$.
\end{proof}

Now to show that the image of $\Gamma$ is in $\mathcal{P}_{r}$, we need some additional properties. 

\begin{Proposition}\label{prop:image}
For all $u\in\{0,\ldots,U-1\}$, the following holds:
\begin{enumerate}
\item We have $0\leq h_{u+1}\leq h_u\leq r-1$.
\item We have $\kappa_u\geq 0$.
\item If $h_u=h_{u+1}$, then $m_{u+1}\geq m_u+2$ and $\kappa_{u+1}\geq \kappa_u$.
\end{enumerate}
\end{Proposition}
\begin{proof}
\begin{enumerate}
\item By Remark~\ref{rem:eta_positive}, $h_u\geq0$ for all $u\geq0$. By definition of $\mathcal{A}_{r}$, $\eta_j^{(0)}+\eta_{j+1}^{(0)}\leq r-1$ for all $j\geq 0$, so $h_0\leq r-1$. Hence, the only thing remaining to show is that for all $u\in\{0,\ldots,U-1\}$, $h_{u+1}\leq h_u$.
By definition of $h_u$, it is enough to show that $\eta^{(u+1)}_j+\eta^{(u+1)}_{j+1}\leq h_u$ for all $j\geq 2u+2$. We consider the three following cases. 
\begin{itemize}
\item If $2u+2\leq j<m_u-1$, then by~\eqref{eq:decalage2}, $\eta^{(u+1)}_j+\eta^{(u+1)}_{j+1}=\eta^{(u)}_{j-2}+\eta^{(u)}_{j-1}$, and by~\eqref{eq:defsu}, we derive
\begin{equation}\label{inegstricte1}
\eta^{(u+1)}_j+\eta^{(u+1)}_{j+1}< h_u.
\end{equation} 
\item If $j\geq m_u$, then by~\eqref{eq:fixed2} and the definition of $h_u$,
$$
\eta^{(u+1)}_j+\eta^{(u+1)}_{j+1}=\eta^{(u)}_{j}+\eta^{(u)}_{j+1}\leq  h_u.
$$
\item If $j=m_u-1$, then $m_u\geq 2u+3$ and by~\eqref{eq:decalage2} and~\eqref{eq:fixed2}, we have $\eta^{(u+1)}_{m_u-1}+\eta^{(u+1)}_{m_u}= \eta^{(u)}_{m_u-3}+\eta^{(u)}_{m_u}$. Now we prove that
\begin{equation}\label{strictmu}
\eta^{(u)}_{m_u-3}<\eta^{(u)}_{m_u-1}.
\end{equation}
Indeed, if we had $\eta^{(u)}_{m_u-3}\geq\eta^{(u)}_{m_u-1}$, then it would yield $\eta^{(u)}_{m_u-3}+\eta^{(u)}_{m_u-2}\geq\eta^{(u)}_{m_u-2}+\eta^{(u)}_{m_u-1}=h_u$ by~\eqref{eq:defsu}, therefore by definition of $h_u$ we would get $\eta^{(u)}_{m_u-3}+\eta^{(u)}_{m_u-2}=h_u$,  contradicting the minimality in~\eqref{eq:defsu}. Therefore
$$
\eta^{(u+1)}_{m_u-1}+\eta^{(u+1)}_{m_u}< \eta^{(u)}_{m_u-1}+\eta^{(u)}_{m_u}\leq\eta^{(u)}_{m_u-1}+\eta^{(u)}_{m_u-2}=h_u\quad\quad\text{(by~\eqref{eq:defsu})},
$$
thus~\eqref{inegstricte1} is still satisfied.
\end{itemize}

\item From the definition of $h_u$, we know that $h_u \geq \eta_{2u}^{(u)}$ and $h_u \geq \eta_j^{(u)}+\eta_{j+1}^{(u)}$ for all $j \geq 2u$. The result follows immediately.

\item Suppose that $h_u=h_{u+1}$. In the proof of Proposition~\ref{prop:image}~(1), we showed that~\eqref{inegstricte1} is satisfied for all $2u+2\leq j<m_u$. Therefore, by~\eqref{eq:defsu}, $m_{u+1}\geq m_u+2$. Now prove the second part:
\begin{align*}
\kappa_{u+1}&=  h_u-\eta_{2u}^{(u)}+ \sum_{j=2u+2}^{m_{u+1}-3}\left[h_u-\left(\eta_j^{(u+1)}+\eta_{j+1}^{(u+1)}\right)\right] &\text{(by~\eqref{eq:decalage2} and~\eqref{eq:defkappa})}\\
&\geq  h_u-\eta_{2u}^{(u)}+\sum_{j=2u+2}^{m_{u}-1}\left[h_u-\left(\eta_j^{(u+1)}+\eta_{j+1}^{(u+1)}\right)\right]&\text{(as $m_{u+1}\geq m_u+2$)}\\
&\geq  h_u-\eta_{2u}^{(u)}+\sum_{j=2u}^{m_{u}-3}\left[h_u-\left(\eta_j^{(u)}+\eta_{j+1}^{(u)}\right)\right]=\kappa_{u}&\text{(by~\eqref{eq:decalage2})}.
\end{align*}
\end{enumerate}
\end{proof}

\begin{Corollary}\label{coro:imGamma}
The image of $\Gamma$ is in $\mathcal{P}_{r}$. More precisely, $\Gamma(\nu) \in \{\mu(s_1,\ldots,s_{r-1})\}\times \mathcal{P}(s_1,\ldots,s_{r-1})$, where for all $j\in\{1,\ldots,r\}$,
$$s_{j} := \min\{u\geq 0: h_u \leq j-1\}.$$

\end{Corollary}
\begin{proof}
By Remark \ref{rem:eta_positive}, $\eta^{(U)}$ is a sequence of non-negative integers. Since $h_U=0$, $\eta^{(U)}_j=0$ for all $j\geq 2U$ by definition. From Proposition \ref{prop:image} (1), we know that $h_u \leq r-1$ for all $u$, so the $s_j$'s are well-defined.

Now check that $\eta^{(U)}=\mu(s_1,\ldots,s_{r-1})$.
Let $u \in \{0, \dots, U-1\}$.
By definition of the $s_j$'s, we have $U=s_{1}\geq \cdots \geq s_{r}=0$, and for all $1\leq j\leq r-1$, 
\begin{equation}\label{equiv}
h_u=j\Leftrightarrow s_{j+1}\leq u < s_j.
\end{equation}

By~\eqref{eq:fixed2}, for all $j<2u+2$, $\eta^{(u+1)}_{j}= \cdots =\eta^{(U-1)}_{j}=\eta^{(U)}_{j}$. Hence
\begin{equation}\label{etas1}
(\eta_{2u}^{(U)},\eta_{2u+1}^{(U)})= (\eta_{2u}^{(u+1)},\eta_{2u+1}^{(u+1)}) = (h_u,0),
\end{equation}
where the second equality follows from~\eqref{eq:def2}.
Therefore by~\eqref{equiv},
$$(\eta_{2u}^{(U)},\eta_{2u+1}^{(U)}) = (j,0)\text{ for all } s_{j+1}\leq u < s_j,$$
which is exactly the multiplicity sequence of $\mu(s_1,\ldots,s_{r-1})$ from Definition~\ref{def:mu}.

Moreover, by~\eqref{equiv} and Proposition~\ref{prop:image}~(2) and~(3), we know that $(\kappa_{s_{j+1}},\ldots,\kappa_{s_{j}-1})$ is a non-decreasing sequence of non-negative integers for all $1\leq j\leq r-1$. Hence $\kappa \in \mathcal{P}(s_1,\ldots,s_{r-1})$. Thus  $\Gamma(\nu) \in \{\mu(s_1,\ldots,s_{r-1})\}\times \mathcal{P}(s_1,\ldots,s_{r-1}) \subset \mathcal{P}_{r}$. 
\end{proof}

Finally, we need two last properties to show that $\Gamma$ is weight- and length-preserving. 

\begin{Proposition}\label{prop:welldefgammaweight}
For all $u\in\{0,\ldots,U-1\}$, the following holds:
\begin{enumerate}
\item the length of $\eta^{(u)}$ equals the length of $\eta^{(u+1)}$, i.e. $\sum_{j\geq 0} \eta^{(u+1)}_{j} = \sum_{j\geq 0} \eta^{(u)}_{j}$,
\item the weight of $\eta^{(u+1)}$ is $\kappa_u$ less than the weight of $\eta_u$, i.e. $\sum_{j\geq 0} j\cdot \eta^{(u+1)}_{j} = -\kappa_u+\sum_{j\geq 0} j\cdot \eta^{(u)}_{j}$.
\end{enumerate}
\end{Proposition}

\begin{proof}
\begin{enumerate}
\item We have
\begin{align*}
\sum_{j\geq 0} \eta_{j}^{(u+1)} &=  \sum_{j=0}^{2u-1} \eta_{j}^{(u)} + h_u+ \sum_{j= 2u+2}^{m_u-1} \eta_{j-2}^{(u)} +  \sum_{j\geq m_u} \eta_{j}^{(u)} &\text{(by~\eqref{eq:fixed2}--\eqref{eq:def2})}\\
&= \sum_{j=0}^{2u-1} \eta_{j}^{(u)} + \eta_{m_u-2}^{(u)}+\eta_{m_u-1}^{(u)}+ \sum_{j= 2u}^{m_u-3} \eta_{j}^{(u)} +  \sum_{j\geq m_u} \eta_{j}^{(u)}&\text{(by~\eqref{eq:defsu})}\\
&=\sum_{j\geq 0} \eta_{j}^{(u)}.
\end{align*} 

\item  We have
\begin{align*}
\sum_{j\geq 0} j\cdot\eta_{j}^{(u+1)} &=  \sum_{j=0}^{2u-1} j\cdot\eta_{j}^{(u)} + 2u\cdot h_u+ \sum_{j= 2u+2}^{m_u-1} j\cdot\eta_{j-2}^{(u)} +  \sum_{j\geq m_u} j\cdot\eta_{j}^{(u)}&\text{(by~\eqref{eq:fixed2}--\eqref{eq:def2})}\\
&= \sum_{j=0}^{2u-1} j\cdot\eta_{j}^{(u)} + 2u\cdot h_u+ \sum_{j= 2u}^{m_u-3} (j+2)\cdot\eta_{j}^{(u)} +  \sum_{j\geq m_u} j\cdot\eta_{j}^{(u)}\\
&=\sum_{j=0}^{m_u-3} j\cdot\eta_{j}^{(u)}+  \sum_{j\geq m_u} j\cdot\eta_{j}^{(u)}+ 2u\cdot h_u-\kappa_u+(m_u-2u-1)h_u-\eta_{m_u-2}^{(u)}&\text{(by~\eqref{eq:defkappa})}\\
&= -\kappa_u+\sum_{j\geq 0} j\cdot\eta_{j}^{(u)}&\text{(by~\eqref{eq:defsu})}.
\end{align*}
\end{enumerate}
\end{proof}

\begin{Corollary}\label{coro:Gammapreserving}
$\Gamma$ preserves the weight and the length, i.e., for all  $\nu \in \mathcal{A}_r$,
$$|\Gamma(\nu)|=|(\mu,\kappa)|=|\nu|,$$
and the length of $\mu$ is equal to the length of $\nu$.
\end{Corollary}
\begin{proof}
By Proposition~\ref{prop:welldefgammaweight} (2),
$$|\Gamma(\nu)|=|\mu|+|\kappa|=|\eta^{(U)}|+|\kappa|=-|\kappa|+|\nu|+|\kappa|= |\nu|,$$ so the weight is preserved. Moreover Proposition~\ref{prop:welldefgammaweight}~(1) implies that the length of $\nu$ is the same as the length of $\mu$, as by definition its multiplicity sequence is $\eta^{(U)}$.
\end{proof}

\subsection{$ \Gamma\circ\Lambda$ is the identity on $\mathcal{P}_r$ \label{sec:GammaoLambda}}
Our goal in this section is to show that $ \Gamma\circ\Lambda$ is the identity map on $\mathcal{P}_r$. 
Let $(\mu(s_1,\ldots,s_{r-1}),\lambda)\in\mathcal{P}_r$, and apply $\Lambda$ to it, using the notations from Section~\ref{sec:Lambda}. Then apply $\Gamma$ with the notations from Section~\ref{sec:Gamma}. We will show that we recover $(\mu(s_1,\ldots,s_{r-1}),\lambda)$, and the following proposition will play a key role in doing that.

\begin{Proposition}\label{prop:inverselambda}
For all $u \in \{0,\ldots,s_{1}\}$, we have
\begin{align}
g_u=\max\{\theta^{(u)}_j+\theta^{(u)}_{j+1}:j\geq 2u\},  \label{eq:g_u} \\
n_u = \min\{j\geq 2u+2:\theta^{(u)}_{j-2}+\theta^{(u)}_{j-1}=g_u\},  \label{eq:n_u}
\end{align}
with the convention that $n_{s_1}=2s_1+2$.
\end{Proposition}
To prove this result, we need the following lemma.

\begin{Lemma}\label{lem:memegu}
For all $u \in \{0,\ldots,s_{1}-1\}$, we have $g_u\geq g_{u+1}$. Moreover $g_u=g_{u+1}$ implies that $n_u\leq n_{u+1}-2$.
\end{Lemma}
\begin{proof}
The fact that $g_u\geq g_{u+1}$ is immediate by definition of $g_u$. 

Now assume that $g_u=g_{u+1}$ and show that $n_u\leq n_{u+1}-2$.
By Proposition~\ref{prop:welldeflambda_s1-1}~(1), we know that $\theta_{2u+1}^{(u+1)}=0$. Thus
\begin{align*}
\sum_{j=2u+2}^{n_{u+1}-2} &\left[ g_u-\left(\theta_j^{(u+1)}+\theta_{j-1}^{(u+1)}\right)\right] = (n_{u+1}-3-2u) g_u -\theta_{n_{u+1}-2}^{(u+1)}-2\sum_{j=2u+2}^{n_{u+1}-3}\theta_j^{(u+1)}\\
&=(n_{u+1}-4-2u) g_{u+1} +\theta_{n_{u+1}-1}^{(u+1)}-2\sum_{j=2u+4}^{n_{u+1}-1}\theta_j^{(u+2)} \qquad \text{(by~\eqref{eq:decalage1} and the first equality of~\eqref{eq:def1})}\\
&=\lambda_{u+1} \qquad \qquad \qquad \qquad \qquad \qquad \qquad \qquad \qquad \qquad \quad \text{(by~\eqref{eq:def1})}\\
&\geq \lambda_u \qquad \qquad \qquad \qquad \qquad \qquad \qquad \qquad \qquad \qquad \qquad \text{(because $g_u=g_{u+1}$ and by definition of $\mathcal{P}_r$).}
\end{align*}
Hence, by~\eqref{eq:defru}, we obtain that $n_u\leq n_{u+1}-2$.
\end{proof}

We are now ready to prove Proposition~\ref{prop:inverselambda}.

\begin{proof}[Proof of Proposition~\ref{prop:inverselambda}]
By Proposition~\ref{prop:welldefnu}, we already know that for all $u$, the set $\{\theta^{(u)}_j+\theta^{(u)}_{j+1}:j\geq 2u\}$ has a maximal element. Moreover, from the first equality of~\eqref{eq:def1}, $g_u = \theta^{(u)}_{n_u-2}+\theta^{(u)}_{n_u-1}$, and from Proposition~\ref{prop:welldeflambda_s1}~(3), for all $j\geq 2u$, $g_u\geq \theta^{(u)}_{j}+\theta^{(u)}_{j+1}$. Hence~\eqref{eq:g_u} is proved.

It only remains to show \eqref{eq:n_u}, i.e. that for all $u \in \{0,\ldots,s_{1}\}$, for all $2u\leq j\leq n_{u}-3$,  $g_{u}>\theta^{(u)}_{j}+\theta^{(u)}_{j+1}$. Let us do it by backward induction on $u$. For $u=s_1$, $n_{s_1}=2s_1+2$, so there is nothing to prove. Now assume the property holds for $u+1$ and show it for $u$.

If $n_u=2u+2$, there is nothing to prove. If $n_u \geq 2u+3$, we distinguish two cases:
\begin{itemize}
\item If $j=n_u-3$, from~\eqref{eq:def1} and the sentence before \eqref{eq:petit2}, we have $\theta_{n_u-1}^{(u)} > \theta_{n_u-1}^{(u+1)}$.
By~\eqref{eq:decalage1}, we deduce
$$\theta^{(u)}_{n_u-3}+\theta^{(u)}_{n_u-2} = \theta_{n_u-1}^{(u+1)} +\theta^{(u)}_{n_u-2} < \theta_{n_u-1}^{(u)} +\theta^{(u)}_{n_u-2} = g_u.$$ 

\item If $2u\leq j\leq n_u-4$, we distinguish two cases (we know from Lemma \ref{lem:memegu} that $g_u \geq g_{u+1}$).

\begin{itemize}
\item If $g_u>g_{u+1}$, then by~\eqref{eq:decalage1},
$$\theta^{(u)}_{j}+\theta^{(u)}_{j+1}=\theta^{(u+1)}_{j+2}+\theta^{(u+1)}_{j+3}\leq g_{u+1} < g_u,$$
where the first inequality follows from Proposition~\ref{prop:welldeflambda_s1}~(3).
\item If $g_u=g_{u+1}$, then again by~\eqref{eq:decalage1},
$$\theta^{(u)}_{j}+\theta^{(u)}_{j+1}=\theta^{(u+1)}_{j+2}+\theta^{(u+1)}_{j+3}< g_{u+1}= g_u,$$
where the inequality follows from the induction hypothesis, as $2u+2\leq j+2\leq n_u-2\leq n_{u+1}-4$ by Lemma ~\ref{lem:memegu}.
\end{itemize}
\end{itemize}
\end{proof}

\begin{Proposition}\label{prop:GrondL}
The map $ \Gamma\circ\Lambda$ is the identity map on $\mathcal{P}_r$. In other words, for all $(\mu(s_1,\ldots,s_{r-1}),\lambda)\in\mathcal{P}_r$, we have
$$\Gamma(\Lambda(\mu(s_1,\ldots,s_{r-1}),\lambda))=(\mu(s_1,\ldots,s_{r-1}),\lambda).$$
\end{Proposition}

\begin{proof}
Let $(\mu(s_1,\ldots,s_{r-1}),\lambda)\in\mathcal{P}_r$. Using the notations of Sections~\ref{sec:Lambda} and~\ref{sec:Gamma} and applying first $\Lambda$ and then $\Gamma$, we first observe that $\eta^{(0)}=\theta^{(0)}$. Then by definition of $h_0$ and Proposition~\ref{prop:inverselambda}, we get $h_0=g_0$. 

\begin{itemize}
\item If $g_0=0$, then $s_1=0$, and the process $\Gamma$ stops at $U=0$. In that case, $\nu = \mu(0,\ldots,0)=\lambda=\emptyset$, and $\Gamma(\Lambda(\mu(0,\ldots,0),\emptyset))=(\mu(0,\ldots,0),\emptyset)$.
\item If $g_0>0$, then $s_1>0$. In that case, $m_0=n_0$ by~\eqref{eq:defsu} and Proposition~\ref{prop:inverselambda}. Therefore, $\eta^{(1)}_j=\theta^{(1)}_j$ for all $j\geq n_0$ by~\eqref{eq:fixed1} and~\eqref{eq:fixed2}, $\eta^{(1)}_j=\theta^{(1)}_j$ for all $2\leq j<n_0$ by ~\eqref{eq:decalage1} and~\eqref{eq:decalage2}, and
$\eta^{(1)}_0 = h_0 = g_0=\theta^{(1)}_0$ and $\eta^{(1)}_1=\theta^{(1)}_1=0$ by Proposition~\ref{prop:welldeflambda_s1-1}~(1) and~\eqref{eq:def2}, hence $\eta^{(1)} = \theta^{(1)}$. Moreover $\kappa_0=\lambda_0$ by~\eqref{eq:def1} and~\eqref{eq:defkappa}.
\end{itemize}

In the same way, we show that $\eta^{(u)}=\theta^{(u)}$ implies  $h_u=g_u$ by definition of $h_u$ and Proposition~\ref{prop:inverselambda}. If $g_u>0$, then $u<s_1$, $\eta^{(u+1)}=\theta^{(u+1)}$, and $\kappa_u=\lambda_u$. Otherwise, $g_u=0$, $u=s_1$ and $\Gamma$ stops at $U=s_1$. Therefore, 
$\Gamma(\Lambda(\mu(s_1,\ldots,s_{r-1}),\lambda))=(\mu(s_1,\ldots,s_{r-1}),\lambda)$.
\end{proof}

\subsection{$\Lambda\circ \Gamma$ is the identity on $\mathcal{A}_{r}$ \label{sec:LambdaoGamma}}
Finally we show that $ \Gamma\circ\Lambda$ is the identity map on $\mathcal{P}_r$. 
Let $\nu\in\mathcal{A}_r$, and apply $\Gamma$ to it, using the notations from Section~\ref{sec:Gamma}. Then apply $\Lambda$ with the notations from Section~\ref{sec:Lambda}. We will show that we recover $\nu$. To do this, we first state a preliminary result.

\begin{Proposition}\label{prop:inversegamma}
Let $u$ be an integer in $\{0,\ldots,U-1\}$. We have  
$$m_u = \min\left\{t\geq 2u+2: \sum_{j=2u+2}^t\left[h_u - \left(\eta^{(u+1)}_{j}+\eta^{(u+1)}_{j-1}\right)\right]\geq \kappa_u\right\}.$$
\end{Proposition}
\begin{proof}
Let $$\varphi\,:\,t\mapsto -\kappa_u+\sum_{j=2u+2}^t\left[h_u - \left(\eta^{(u+1)}_{j}+\eta^{(u+1)}_{j-1}\right)\right].$$
We want to show that $$m_u = \min\left\{t\geq 2u+2: \varphi(t)\geq0\right\}.$$

First, we treat the case $m_u=2u+2$. By the definitions of $h_u$ and $m_u$, we know that $\eta_{2u}^{(u)}+\eta_{2u+1}^{(u)}=h_u$ and $\eta_{2u+1}^{(u)}+\eta_{2u+2}^{(u)}\leq h_u$.
Thus
\begin{align*}
\varphi(2u+2) &= -\kappa_u + h_u - \left(\eta^{(u+1)}_{2u+1}+\eta^{(u+1)}_{2u+2}\right)\\
&=h_u - \left(\eta^{(u)}_{2u+1}+\eta^{(u)}_{2u+2}\right) &\text{(by~\eqref{eq:fixed2},~\eqref{eq:defkappa} and~\eqref{eq:def2})}\\
&\geq 0.
\end{align*}

Now turn to the case $m_u\geq 2u+3$.
Note that, for all $j\geq 2u+3$, 
$\eta^{(u+1)}_{j}+\eta^{(u+1)}_{j-1}\leq h_{u+1}\leq h_u$, where the first inequality follows from the definition of $h_{u+1}$ and the second from Proposition~\ref{prop:image}~(1). Thus the function  
$\varphi$
is non-decreasing on $\{2u+2,2u+3,\dots\}$. Hence, we only have to show that $\varphi(m_u-1)<0$ and $\varphi(m_u)\geq 0$. 

First, we have
\begin{align*}
\varphi(m_u-1)&=-\kappa_u+\sum_{j=2u+2}^{m_u-1}\left[h_u - \left(\eta^{(u+1)}_{j}+\eta^{(u+1)}_{j-1}\right)\right]\\
&= -\kappa_u+\eta^{(u)}_{m_u-2}+\eta^{(u)}_{m_u-3}-\eta^{(u)}_{2u}-\eta_{2u+1}^{(u+1)}+\sum_{j=2u}^{m_u-3}\left[h_u - \left(\eta^{(u)}_{j}+\eta^{(u)}_{j+1}\right)\right]&\text{(by ~\eqref{eq:decalage2})},
\end{align*}
which by~\eqref{eq:def2},~\eqref{eq:defkappa} and~\eqref{eq:defsu} yields
\begin{equation}\label{phi}
\varphi(m_u-1)= \eta^{(u)}_{m_u-3}-\eta^{(u)}_{m_u-1}.
\end{equation}
By~\eqref{strictmu}, this implies $\varphi(m_u-1)<0$.

Second, we have
\begin{align*}
\varphi(m_u) &= \varphi(m_u-1)+ h_u- \left(\eta^{(u+1)}_{m_u}+\eta^{(u+1)}_{m_u-1}\right)\\
&= h_u-\eta^{(u)}_{m_u-1} -\eta^{(u)}_{m_u} &\text{(by~\eqref{phi},~\eqref{eq:fixed2}, and~\eqref{eq:decalage2})}\\
&\geq 0 &\text{(by~\eqref{eq:defsu})}.
\end{align*}
The proposition is proved.
\end{proof}

\begin{Proposition}\label{prop:LrondG}
The map $\Lambda\circ \Gamma$ is the identity map on $\mathcal{A}_{r}$. In other words, for all $\nu\in\mathcal{A}_r$, we have
$$\Lambda(\Gamma(\nu))=\nu.$$
\end{Proposition}

\begin{proof}
Let $\nu\in\mathcal{A}_r$. Using the notations of Sections~\ref{sec:Lambda} and~\ref{sec:Gamma} and applying first $\Gamma$ and then $\Lambda$, we first observe that $\theta^{(s_1)}=\eta^{(s_1)}$, as by Corollary~\ref{coro:imGamma} we have $U=s_1$. Therefore by the definitions of $g_{s_1}$ and $U$, we obtain $h_{s_1}=0=g_{s_1}$. 

\begin{itemize}
\item If $s_1=0$, then $\nu=\kappa=\emptyset$ and $\Lambda(\Gamma(\emptyset)))=\Lambda(\mu(0,\ldots,0),\emptyset)=\emptyset$.

\item If $s_1>0$, we prove by backward induction on $0\leq u \leq s_1$ that   $\theta^{(u)}=\eta^{(u)}$.
Assume that for some $0\leq u<s_1$, $\theta^{(u+1)}=\eta^{(u+1)}$.
By~\eqref{etas1}, we have 
$$\left(\eta^{(u+1)}_{2u},\eta^{(u+1)}_{2u+1}\right)=(h_u,0)=\left(\eta^{(s_1)}_{2u},\eta^{(s_1)}_{2u+1}\right),$$
 and $g_u=h_u$ by~\eqref{equiv}.
  Hence, by Proposition~\ref{prop:inversegamma} and ~\eqref{eq:defru}, $n_u=m_u$. Therefore, $\theta^{(u)}_j=\eta^{(u)}_j$ for all $j\geq m_u$ by~\eqref{eq:fixed1} and~\eqref{eq:fixed2}, $\theta^{(u)}_j=\eta^{(u)}_j$ for all $2u\leq j<m_u-2$ by~\eqref{eq:decalage1} and~\eqref{eq:decalage2}, 
$\theta_{m_u-1}^{(u)} = \eta^{(u)}_{m_u-1}$
by~\eqref{eq:def1} and~\eqref{eq:defkappa}, and finally 
$\theta_{m_u-2}^{(u)} = \eta^{(u)}_{m_u-2}$ by~\eqref{eq:def1} and~\eqref{eq:def2}.
Hence, $\theta^{(u)}=\eta^{(u)}$.
\end{itemize}
Thus in particular $\theta^{(0)}=\eta^{(0)}$, i.e. $\Lambda(\Gamma(\nu))=\nu$.
\end{proof}

\section{Proof of Corollaries~\ref{coro:GF}--\ref{coro:list}}\label{sec:cor}

To prove our corollaries, we need to show that our maps $\Lambda$ and $\Gamma$ send the desired subsets of $\mathcal{P}_r$ from Definition~\ref{def:LHSsets} and the ones of $\mathcal{A}_r$ from Proposition~\ref{prop:TUfreq} and Definition~\ref{def:ABB} to the appropriate images.

\subsection{Maps induced by $\Lambda$}

We start with a preliminary result. 

\begin{Proposition}\label{prop:inducedlambda}
Let $0\leq i\leq r-1$. Let $(\mu(s_1,\ldots,s_{r-1}),\lambda)\in\mathcal{P}_r$, and apply $\Lambda$ to it, using the notations from Section~\ref{sec:Lambda}. Then the following holds.
\begin{enumerate} 
\item For all $u \in \{0,\ldots,s_1-1\}$, $\displaystyle (n_u-2)\theta^{(u)}_{n_u-2}+ (n_u-1)\theta^{(u)}_{n_u-1} \equiv \lambda_u \mod 2$.
\item If $\lambda \in \mathcal{P}_{i,r}(s_1,\ldots,s_{r-1})$, then for all $u \in \{0,\ldots,s_1\}$, $\theta^{(u)}_{2u}\leq i$. 
\item If $\lambda \in \mathcal{Q}_{i,r}(s_1,\ldots,s_{r-1})$,  then for all $u \in \{0,\ldots,s_1\}$,
$\theta^{(u)}_{2u}=0$ and $\theta^{(u)}_{2u+1}\leq i$.
\end{enumerate}
\end{Proposition}

\begin{proof}
\begin{enumerate}
\item Let $u \in \{0,\ldots,s_1-1\}$. By ~\eqref{eq:def1}, we have
\begin{align*}(n_u-2)\theta^{(u)}_{n_u-2}+ (n_u-1)\theta^{(u)}_{n_u-1}&= (n_u-2)g_u+ \theta^{(u)}_{n_u-1}\\
&= (n_u-2)g_u+\theta_{n_u-1}^{(u+1)}+\lambda_u-\sum_{j=2u+2}^{n_u-1} \left[g_u-\left(\theta_j^{(u+1)}+\theta_{j-1}^{(u+1)}\right)\right]\\
&= -2ug_u+\lambda_u+\theta_{2u+1}^{(u+1)}+2\sum_{j=2u+2}^{n_u-2}\theta_{j}^{(u+1)}\\
&\equiv \lambda_u+ \theta^{(u+1)}_{2u+1} \mod 2,
\end{align*}
and we conclude by using Proposition~\ref{prop:welldeflambda_s1-1}~(1).

\item Assume that $\lambda \in \mathcal{P}_{i,r}(s_1,\ldots,s_{r-1})$. We prove the result by backward induction on $u \in \{0,\ldots,s_1\}$. First, by~\eqref{eq:rappel}, $\theta^{(s_1)}_{2s_1}=0\leq i$.
Now assume that $\theta^{(u+1)}_{2u+2}\leq i$, and show that $\theta^{(u)}_{2u}\leq i$.
\begin{itemize}
\item If $n_u>2u+2$, then, by~\eqref{eq:decalage1}, $\theta^{(u)}_{2u}=\theta^{(u+1)}_{2u+2}\leq i$.
\item If $n_u=2u+2$, then, by~\eqref{eq:def1} and Proposition~\ref{prop:welldeflambda_s1-1}~(1), we get $\theta^{(u)}_{2u}=g_u-\theta^{(u+1)}_{2u+1}-\lambda_u=g_u-\lambda_u$. Recall that by definition of $g_u$, we have $s_{g_u+1}\leq u<s_{g_u}$. Thus by Definition \ref{def:LHSsets}, 
$\lambda_u\geq \lambda_{s_{g_u+1}}\geq g_u-i$. Therefore
$\theta^{(u)}_{2u}=g_u-\lambda_u \leq i$.
\end{itemize}

\item Similarly, assume that $\lambda \in \mathcal{Q}_{i,r}(s_1,\ldots,s_{r-1})$. We prove again the result by  backward induction on $u \in \{0,\ldots,s_1\}$. First, by~\eqref{eq:rappel}, $\theta^{(s_1)}_{2s_1}=\theta^{(s_1)}_{2s_1+1}=0\leq i$. Now assume that $\theta^{(u+1)}_{2u+2}=0$ and $\theta^{(u+1)}_{2u+3}\leq i$, and show that $\theta^{(u)}_{2u}=0$ and $\theta^{(u)}_{2u+1}\leq i$. We distinguish three cases.
\begin{itemize}
\item If $n_u>2u+3$, then by~\eqref{eq:decalage1}, $\theta^{(u)}_{2u}=\theta^{(u+1)}_{2u+2}=0$ and $\theta^{(u)}_{2u+1}=\theta^{(u+1)}_{2u+3}\leq i$.
\item If $n_u=2u+3$, then, by~\eqref{eq:decalage1}, $\theta^{(u)}_{2u}=\theta^{(u+1)}_{2u+2}=0$, and
\begin{align*}
\theta^{(u)}_{2u+1} &=2g_u-\lambda_u-2\theta_{2u+2}^{(u+1)}-\theta_{2u+1}^{(u+1)} &\text{by~\eqref{eq:def1}}
\\&= 2g_u-\lambda_u &\text{by~Proposition~\ref{prop:welldeflambda_s1-1}~(1).}
\end{align*}
Again, as $s_{g_u+1}\leq u<s_{g_u}$, we derive by Definition \ref{def:LHSsets} that
$\lambda_u\geq \lambda_{s_{g_u+1}}\geq g_u+\max\{g_u-i,0\}$. Therefore,
$$\theta^{(u)}_{2u+1}=2g_u-\lambda_u \leq g_u-\max\{g_u-i,0\} = \min\{g_u,i\}\leq i.$$
\item If $n_u=2u+2$, then by~\eqref{eq:def1}, $\theta^{(u)}_{2u}=g_u-\lambda_u$ and $\theta^{(u)}_{2u+1}=g_u-\theta^{(u)}_{2u}$. 
As above, $\lambda_u\geq g_u+ \max\{g_u-i,0\}$, thus by Proposition~\ref{prop:welldeflambda_s1-1}~(1),
$$g_u - \theta_{2u+2}^{(u+1)} -\theta_{2u+1}^{(u+1)} =g_u \leq \lambda_u - \max\{g_u-i,0\},$$
which by~\eqref{eq:defru} implies that $n_u=2u+2$ only if $g_u\leq i$ and $\lambda_u=g_u$. Therefore
$\theta^{(u)}_{2u}=g_u-\lambda_u=0$ and $\theta^{(u)}_{2u+1}=g_u\leq i$. 
\end{itemize}
\end{enumerate}
\end{proof}

We can now show that the images by $\Lambda$ of the subsets of $\mathcal{P}_r$ from Definition~\ref{def:LHSsets} are included in the desired subsets of $\mathcal{A}_r$ from Proposition~\ref{prop:TUfreq} and Definition~\ref{def:ABB}.
\begin{Corollary}\label{coro:inducelambda}
For all $r$ and $i$ integers such that $r\geq 2$ and $0\leq i \leq r-1$, we have 
\begin{enumerate}
\item $\Lambda(\mathcal{P}_{i,r})\subset\mathcal{A}_{i,r}$,
\item $\Lambda(\mathcal{Q}_{i,r})\subset\mathcal{T}_{i+1,r}$,
\item $\Lambda(\mathcal{R}_{i,r})\subset\mathcal{B}_{i,r}$,
\item $\Lambda(\tilde{\mathcal{R}}_{i,r})\subset\tilde{\mathcal{B}}_{i,r}$,
\item $\Lambda(\mathcal{S}_{i,r})\subset\mathcal{U}_{i+1,r}$,
\item $\Lambda(\tilde{\mathcal{S}}_{i,r})\subset\tilde{\mathcal{U}}_{i+1,r}$.
\end{enumerate}
\end{Corollary}

\begin{proof}
Let $(\mu(s_1,\ldots,s_{r-1}),\lambda) \in \mathcal{P}_r$, let $\nu=(f_j)_{j\geq0}$ denote its image by $\Lambda$, and use the notations from Section~\ref{sec:Lambda}. Recall that $f_j=\theta^{(0)}_j$ for all $j \geq 0$.

\begin{enumerate}
\item If $(\mu(s_1,\ldots,s_{r-1}),\lambda) \in \mathcal{P}_{i,r}$, 
we have $f_j+f_{j+1}\leq r-1$ for all $j$ because $\nu\in\mathcal{A}_{r}$. Moreover, using Proposition~\ref{prop:inducedlambda}~(2) with $u=0$, we obtain $f_0=\theta^{(0)}_0\leq i$, therefore $\nu\in\mathcal{A}_{i,r}$.

\item If $(\mu(s_1,\ldots,s_{r-1}),\lambda) \in \mathcal{Q}_{i,r}$, similarly we obtain  $f_0=0$ and $\theta^{(0)}_1\leq i$ by Proposition~\ref{prop:inducedlambda}~(3) with $u=0$. Therefore $\nu \in \mathcal{T}_{i+1,r}$.
\end{enumerate}

For the proof of (3)--(6), we distinguish two cases depending on the value of $s_{r-1}$.
\begin{itemize}
\item If $s_{r-1}=0$, then by definition $g_u\leq r-2$ for all $u\in\{0,\dots,s_1\}$. This implies by Proposition~\ref{prop:inverselambda} that $\theta^{(u)}_{j}+\theta^{(u)}_{j+1}<r-1$ for all $j\geq2u$. In particular for $u=0$, this gives $f_{j}+f_{j+1}<r-1$ for all $j\geq0$. Hence the additional conditions of the type ``$f_j+f_{j+1}=r-1$ only if \dots" in the sets $\mathcal{B}_{i,r},\ \tilde{\mathcal{B}}_{i,r},\ \mathcal{U}_{i+1,r},\ \tilde{\mathcal{U}}_{i+1,r}$ are void and there is nothing else to prove than $\Lambda(\mathcal{P}_{i,r})\subset \mathcal{A}_{i,r}$ and $\Lambda(\mathcal{Q}_{i,r})\subset \mathcal{T}_{i+1,r}$, which we just did.

\item If $s_{r-1}>0$, we need to examine for which integers $j$ we have $f_j+f_{j+1}=r-1$, or equivalently 
\begin{equation}\label{r-1}
\theta_{j}^{(0)}+\theta_{j+1}^{(0)}=r-1.
\end{equation}
First, by definition of $g_u$, we know that $g_u \leq r-2$ for $u \geq s_{r-1}$, and
$g_0 = g_1 = \cdots =g_{s_{r-1}-1}=r-1$. Thus by Lemma~\ref{lem:memegu}, $(n_u-2u)_{u=0}^{s_{r-1}-1}$ is a non-decreasing sequence of non-negative integers, and by \eqref{eq:g_u}, we know that $\theta^{(u)}_j+\theta^{(u)}_{j+1} \leq r-2$ for all $j \geq 2u$ unless $u \in \{0, \dots , s_{r-1}-1\}$. 
Therefore, by~\eqref{eq:fixed1}, for all $j\geq n_{s_{r-1}-1}\geq 2s_{r-1}$,
$$\theta_{j}^{(0)}+\theta_{j+1}^{(0)}= \cdots =\theta_{j}^{(s_{r-1}-1)}+\theta_{j+1}^{(s_{r-1}-1)}= \theta_{j}^{(s_{r-1})}+\theta_{j+1}^{(s_{r-1})} \leq r-2.$$

Hence~\eqref{r-1} may only be satisfied if $j< n_{s_{r-1}-1}$.

For all $u \in \{0, \dots , s_{r-1}-1\}$, by~\eqref{eq:fixed1}, we have $\theta_{j}^{(0)}=\theta_{j}^{(u)}$ for all $j \geq n_{u-1}$, with the convention that $n_{-1}=0$. 
Thus for  $n_{u-1}\leq j\leq n_{u}-3$, 
$$\theta_{j}^{(0)}+\theta_{j+1}^{(0)} = \theta_{j}^{(u)}+\theta_{j+1}^{(u)}<g_{u}=r-1,$$ 
where the inequality follows from~\eqref{eq:n_u}. Hence~\eqref{r-1} may only be satisfied if
$$j\in \{n_u-2,n_u-1: 0\leq u<s_{r-1}\}.$$

For $j= n_u-2\geq n_{u-1}$, we obtain
$$\theta_{n_u-2}^{(0)}+\theta_{n_u-1}^{(0)}=\theta_{n_u-2}^{(u)}+\theta_{n_u-1}^{(u)}=g_u=r-1,$$
where the second equality follows from~\eqref{eq:def1}. Hence~\eqref{r-1} is satisfied for all $j \in \{n_u-2: 0\leq u<s_{r-1}\}$.

Now if~\eqref{r-1} is satisfied for $j=n_u-1$, as we know that it is also satisfied for $j=n_u-2$, we derive $\theta_{n_u}^{(0)}=\theta_{n_u-2}^{(0)}=r-1-\theta_{n_u-1}^{(0)}$. Therefore,
$$(n_u-1)\cdot \theta_{n_u-1}^{(0)}+ n_u\cdot \theta_{n_u}^{(0)}\equiv (n_u-2)\cdot \theta_{n_u-2}^{(0)}+ (n_u-1)\cdot \theta_{n_u-1}^{(0)} \equiv\lambda_u \mod 2,$$
by Proposition~\ref{prop:inducedlambda}~(1). Thus, for all $j\geq0$,
\begin{equation}\label{r-1ccl}
f_{j}+f_{j+1}=r-1\Rightarrow \exists u\in\{0,\dots,s_{r-1}-1\},\;j\cdot f_{j} + (j+1)\cdot f_{j+1} \equiv \lambda_u \mod 2.
\end{equation}

If $(\mu(s_1,\dots,s_{r-1}),\lambda)$ belongs to $\mathcal{R}_{i,r}$ (resp. $\tilde{\mathcal{R}}_{i,r}$), we have $\nu\in\mathcal{A}_{i,r}$, because $\mathcal{R}_{i,r}$ and $\tilde{\mathcal{R}}_{i,r}$ are subsets of $\mathcal{P}_{i,r}$. Using~\eqref{r-1ccl}, we deduce that $\nu\in\mathcal{B}_{i,r}$ (resp. $\tilde{\mathcal{B}}_{i,r}$), and (3) and (4) are proved.

Similarly, if $(\mu(s_1,\dots,s_{r-1}),\lambda)$ belongs to $\mathcal{S}_{i,r}$ (resp. $\tilde{\mathcal{S}}_{i,r}$), we have $\nu\in\mathcal{T}_{i+1,r}$, because $\mathcal{S}_{i,r}$ and $\tilde{\mathcal{S}}_{i,r}$ are subsets of $\mathcal{Q}_{i,r}$. Using~\eqref{r-1ccl}, we deduce that $\nu\in\mathcal{U}_{i+1,r}$ (resp. $\tilde{\mathcal{U}}_{i+1,r}$), and (5) and (6) are proved.
\end{itemize}
\end{proof}

\subsection{Maps induced by $\Gamma$}

Again we start with a preliminary result. 

\begin{Proposition}\label{prop:inducedgamma}
Let $0\leq i\leq r-1$. Let $\nu\in\mathcal{A}_r$, and apply $\Gamma$ to it, using the notations from Section~\ref{sec:Gamma}. Then the following holds. 
\begin{enumerate}
\item For all $u \in \{0,\ldots,U-1\}$,
$(m_u-2)\cdot\eta^{(u)}_{m_u-2}+ (m_u-1)\cdot\eta^{(u)}_{m_u-1} \equiv \kappa_u\mod 2.$
\item If $\nu \in \mathcal{A}_{i,r}$, then for all $u \in \{0,\ldots,U\}$, $\eta^{(u)}_{2u}\leq i$.
\item If $\nu \in \mathcal{T}_{i+1,r}$, then for all $u \in \{0,\ldots,U\}$, $\eta^{(u)}_{2u}=0$ and $\eta^{(u)}_{2u+1}\leq i$.
\end{enumerate}
\end{Proposition}

\begin{proof}
\begin{enumerate}
\item Let $u \in \{0,\ldots,U-1\}$. By~\eqref{eq:defsu} and~\eqref{eq:defkappa}, we have
\begin{align*}(m_u-2)\cdot\eta^{(u)}_{m_u-2}+ (m_u-1)\cdot\eta^{(u)}_{m_u-1}&= (m_u-1)h_u- \eta^{(u)}_{m_u-2}\\
&= (m_u-2)h_u- \eta^{(u)}_{m_u-2}+\kappa_u+\eta_{2u}^{(u)}- \sum_{j=2u}^{m_u-3}\left[h_u-\left(\eta_j^{(u)}+\eta_{j+1}^{(u)}\right)\right]\\
&= 2uh_u+\kappa_u+ 2\sum_{j=2u}^{m_u-3}\eta_j^{(u)}\\
&\equiv \kappa_u \mod 2.
\end{align*}
\item Assume that $\nu \in \mathcal{A}_{i,r}$, then its multiplicity sequence $(f_j)_{j\geq0}$ satisfies $f_0\leq i$. We prove the result by induction on $u\in \{0,\ldots,U\}$, and first observe that $\eta^{(0)}_{0}=f_0\leq i$. Now assume that $\eta^{(u)}_{2u}\leq i$ for $u<U$, and show that $\eta^{(u+1)}_{2u+2}\leq i$.
\begin{itemize}
\item If $m_u>2u+2$, then by~\eqref{eq:decalage2}, $\eta^{(u+1)}_{2u+2}=\eta^{(u)}_{2u}\leq i$.
\item If $m_u=2u+2$, then by~\eqref{eq:fixed2}, $\eta^{(u+1)}_{2u+2}=\eta^{(u)}_{2u+2}$. By definition of $h_u$ and~\eqref{eq:defsu},
$$\eta^{(u)}_{2u+2}+\eta^{(u)}_{2u+1}\leq h_u = \eta^{(u)}_{2u}+\eta^{(u)}_{2u+1},$$
so that $\eta^{(u+1)}_{2u+2}\leq \eta^{(u)}_{2u}\leq i$.
\end{itemize}
\item Similarly, assume that $\nu \in \mathcal{T}_{i+1,r}$ with multiplicity sequence $(f_j)_{j\geq0}$. We prove again the result by induction on $u\in \{0,\ldots,U\}$, starting with $\eta^{(0)}_{0}=f_0=0$ and $\eta^{(0)}_{1}=f_1\leq i$. Now assume that $\eta^{(u)}_{2u}=0$ and $\eta^{(u)}_{2u+1}\leq i$ for $u<U$, and show that  $\eta^{(u+1)}_{2u+2}=0$ and $\eta^{(u+1)}_{2u+3}\leq i$. We distinguish three cases.
\begin{itemize}
\item If $m_u>2u+3$, then by~\eqref{eq:decalage2}, $\eta^{(u+1)}_{2u+2}=\eta^{(u)}_{2u}=0$ and $\eta^{(u+1)}_{2u+3}=\eta^{(u)}_{2u+1}\leq i$.
\item If $m_u=2u+3$, then by~\eqref{eq:decalage2}, $\eta^{(u+1)}_{2u+2}=\eta^{(u)}_{2u}=0$. Moreover by ~\eqref{eq:fixed2}, $\eta^{(u+1)}_{2u+3}=\eta^{(u)}_{2u+3}$. By definition of $h_u$ and~\eqref{eq:defsu},
$$\eta^{(u)}_{2u+2}+\eta^{(u)}_{2u+3}\leq h_u = \eta^{(u)}_{2u+1}+\eta^{(u)}_{2u+2}.$$
Hence, $\eta^{(u+1)}_{2u+3}\leq \eta^{(u)}_{2u+1}\leq i$.
\item If $m_u=2u+2$, then by Remark~\ref{rem:eta_positive}, $\eta_{2u+2}^{(u)}\geq0$. By definition of $h_u$ and~\eqref{eq:defsu}, we also have 
$$\eta^{(u)}_{2u+2}+\eta^{(u)}_{2u+1}\leq h_u = \eta^{(u)}_{2u}+\eta^{(u)}_{2u+1},$$
so that $0\leq \eta_{2u+2}^{(u)} \leq \eta_{2u}^{(u)}=0$. Therefore $\eta_{2u+2}^{(u)}=0$, and then 
$$\eta_{2u+3}^{(u)}\leq h_u = \eta_{2u+1}^{(u)}\leq i.$$
Finally by~\eqref{eq:fixed2}, $\eta^{(u+1)}_{2u+2}=\eta^{(u)}_{2u+2}=0$ and $\eta^{(u+1)}_{2u+3}=\eta^{(u)}_{2u+3}\leq i$.
\end{itemize}
\end{enumerate}
\end{proof}

We can now show that the images by $\Gamma$ of the subsets of $\mathcal{A}_r$ defined in Proposition~\ref{prop:TUfreq} and Definition~\ref{def:ABB} are included in the desired subsets of $\mathcal{P}_r$ from Definition~\ref{def:LHSsets}.

\begin{Corollary}\label{coro:inducegamma}
For all $r$ and $i$ integers such that $r\geq 2$ and $0\leq i \leq r-1$, we have 
\begin{enumerate}
\item $\Gamma(\mathcal{A}_{i,r})\subset\mathcal{P}_{i,r}$,
\item $\Gamma(\mathcal{T}_{i+1,r})\subset\mathcal{Q}_{i,r}$,
\item $\Gamma(\mathcal{B}_{i,r})\subset\mathcal{R}_{i,r}$,
\item $\Gamma(\tilde{\mathcal{B}}_{i,r})\subset\tilde{\mathcal{R}}_{i,r}$,
\item $\Gamma(\mathcal{U}_{i+1,r})\subset\mathcal{S}_{i,r}$,
\item $\Gamma(\tilde{\mathcal{U}}_{i+1,r})\subset\tilde{\mathcal{S}}_{i,r}$.
\end{enumerate}
\end{Corollary}

\begin{proof}
Let $\nu\in \mathcal{A}_r$, let $(\mu(s_1,\ldots,s_{r-1}),\kappa)$ denote its image by $\Gamma$, and use the notations from Section~\ref{sec:Gamma}. 

\begin{enumerate}
\item If $\nu \in \mathcal{A}_{i,r}$, first observe that for all $0\leq u \leq U-1$, by~\eqref{eq:defkappa} and by definition of $h_u$,
$$
\kappa_u =   h_u-\eta_{2u}^{(u)} +\sum_{j=2u}^{m_u-3}\left[h_u-\left(\eta_j^{(u)}+\eta_{j+1}^{(u)}\right)\right]\geq h_u-\eta_{2u}^{(u)}.
$$
As $\nu \in \mathcal{A}_{i,r}$, Proposition~\ref{prop:inducedgamma}~(2) then implies $\kappa_u\geq h_u-i$. By~\eqref{equiv}, for all $1\leq j\leq r-1$, $h_{s_{j+1}}=j$. Hence $\kappa_{s_{j+1}}\geq j-i$. Thus $\kappa \in \mathcal{P}_{i,r}(s_1,\ldots,s_{r-1})$, so that $\Gamma(\nu)\in \mathcal{P}_{i,r}$.
\item Suppose now that $\nu \in \mathcal{T}_{i+1,r}$ and let $u$ be such that $0\leq u\leq U-1$.
\begin{itemize}
\item If $m_u\geq2u+3$, then 
\begin{align*}
\kappa_u&\geq 2h_u-2\eta_{2u}^{(u)}-\eta_{2u+1}^{(u)}&\text{(by~\eqref{eq:defkappa} and by definition of $h_u$)}\\
&= h_u+ h_u-\eta_{2u+1}^{(u)} &\text{(by Proposition~\ref{prop:inducedgamma}~(3))}\\
&\geq h_u+ \max\{h_u-i,0\} &\text{(by Proposition~\ref{prop:inducedgamma}~(3))}.
\end{align*}
\item If $m_{u}=2u+2$, then by~\eqref{eq:defkappa}, $\kappa_u= h_u-\eta_{2u}^{(u)}$. Thus by Proposition~\ref{prop:inducedgamma}~(3), $\kappa_u= h_u$. As by~\eqref{eq:defsu} we have $h_u=\eta_{2u}^{(u)}+\eta_{2u+1}^{(u)}$, we derive using Proposition~\ref{prop:inducedgamma}~(3) that $h_u\leq i$, and therefore again $\kappa_u\geq h_u+\max\{h_u-i,0\}$.
\end{itemize}
Using as before $h_{s_{j+1}}=j$ for all $1\leq j\leq r-1$, the above discussion yields
$$\kappa_{s_{j+1}}\geq j+ \max\{j-i,0\},$$
 so that $\kappa \in \mathcal{Q}_{i,r}(s_1,\ldots,s_{r-1})$ and therefore $\Gamma(\nu)\in \mathcal{Q}_{i,r}$.
\end{enumerate}
 
For the proof of (3)--(6), we distinguish two cases depending on the value of $s_{r-1}$.
\begin{itemize}
\item If $s_{r-1}=0$, then $\{0,\dots,s_{r-1}-1\}$ is empty, so the  parity conditions involved  in the sets $\mathcal{R}_{i,r}$, $\tilde{\mathcal{R}}_{i,r}$, $\mathcal{S}_{i,r}$, and $\tilde{\mathcal{S}}_{i,r}$ are void, and there is nothing else to prove than $\Gamma(\mathcal{A}_{i,r})\subset \mathcal{P}_{i,r}$ and $\Gamma(\mathcal{T}_{i+1,r})\subset \mathcal{Q}_{i,r}$, which we just did.
\item If $s_{r-1}>0$, we need to examine the parity of $\kappa_{u}$ for the integers $u$ such that $0\leq u\leq s_{r-1}-1$. By~\eqref{equiv} with $j=r-1$,
$$h_0=h_1=\dots=h_{s_{r-1}-1}=r-1.$$
Therefore Proposition~\ref{prop:image}~(3) yields $m_u+2\leq m_{u+1}$ for all $0\leq u\leq s_{r-1}-1$. By repeatedly using~\eqref{eq:fixed2}, this implies in particular $\eta_{j}^{(u+1)}=\eta_{j}^{(0)}$ for all $j\geq m_u$ and all these integers $u$. As $m_u\geq m_{u-1}+2$, this yields $\eta^{(0)}_{m_u-2}=\eta^{(u)}_{m_u-2}$, $\eta^{(0)}_{m_u-1}=\eta^{(u)}_{m_u-1}$, and
$$\eta^{(0)}_{m_u-2}+\eta^{(0)}_{m_u-1}=\eta^{(u)}_{m_u-2}+\eta^{(u)}_{m_u-1}=h_u=r-1,$$
where the second equality follows from~\eqref{eq:defsu}. Therefore for $j=m_u-2$, we derive
\begin{align*}
j\cdot \eta^{(0)}_j+(j+1)\cdot\eta^{(0)}_{j+1}&= (m_u-2)\cdot\eta^{(0)}_{m_u-2}+ (m_u-1)\cdot\eta^{(0)}_{m_u-1} \\
&= (m_u-2)\cdot\eta^{(u)}_{m_u-2}+ (m_u-1)\cdot\eta^{(u)}_{m_u-1} \\
&\equiv\kappa_u \mod2 &\text{(by Proposition~\ref{prop:inducedgamma}~(1))}.
\end{align*} 
Finally, for $u\in\{0,\dots,s_{r-1}-1\}$,
\begin{equation}\label{r-1ccl'}
\kappa_u\equiv j\cdot \eta_{j}^{(0)} + (j+1)\cdot\eta_{j+1}^{(0)} \mod 2\quad\mbox{for some $j$ satisfying}\quad \eta_{j}^{(0)}+\eta_{j+1}^{(0)}=r-1.
\end{equation}

If $\nu$ belongs to $\mathcal{B}_{i,r}$ (resp. $\tilde{\mathcal{B}}_{i,r}$), then $\Gamma(\nu)\in\mathcal{P}_{i,r}$, because $\mathcal{B}_{i,r}$ and $\tilde{\mathcal{B}}_{i,r}$ are subsets of $\mathcal{A}_{i,r}$. Using~\eqref{r-1ccl'}, we derive $\Gamma(\nu)\in\mathcal{R}_{i,r}$ (resp. $\Gamma(\nu)\in\tilde{\mathcal{R}}_{i,r}$), which proves (3) and (4).

Similarly, if $\nu$ belongs to $\mathcal{U}_{i+1,r}$ (resp. $\tilde{\mathcal{U}}_{i+1,r}$), then $\Gamma(\nu)\in\mathcal{Q}_{i,r}$, because $\mathcal{U}_{i+1,r}$ and $\tilde{\mathcal{U}}_{i+1,r}$ are subsets of $\mathcal{T}_{i+1,r}$.
Using~\eqref{r-1ccl'}, we derive $\Gamma(\nu)\in\mathcal{S}_{i,r}$ (resp. $\Gamma(\nu)\in\tilde{\mathcal{S}}_{i,r}$), which proves (5) and (6).
\end{itemize}

\end{proof}

\subsection{Proof of Corollary~\ref{coro:GF}}

Using the bijection between $\mathcal{Q}_{i-1,r}$ and $\mathcal{T}_{i,r}$ induced in Corollaries~\ref{coro:inducelambda}~(2) and~\ref{coro:inducegamma}~(2) by our maps $\Lambda$ and $\Gamma$, we obtain for $1\leq i\leq r$:
$$
\sum_{n\geq 0} T_{i,r}(n)q^n=\sum_{(\mu,\lambda) \in \mathcal{Q}_{i-1,r}} q^{|(\mu,\lambda)|},$$
and this gives the desired~\eqref{eq:AGW} by~\eqref{AGriter}. Similarly,  using the bijection between $\mathcal{S}_{i-1,r}$ and $\mathcal{U}_{i,r}$ induced in Corollaries~\ref{coro:inducelambda}~(5) and~\ref{coro:inducegamma}~(5) by our maps $\Lambda$ and $\Gamma$, we get for $1\leq i\leq r$:
$$
\sum_{n\geq 0} U_{i,r}(n)q^n=\sum_{(\mu,\lambda) \in \mathcal{S}_{i-1,r}} q^{|(\mu,\lambda)|},
$$
and this gives the desired~\eqref{eq:BriW} by~\eqref{Briter}.

\subsection{Proof of Corollary~\ref{coro:newbressoud}}

Using the bijection between $\tilde{\mathcal{S}}_{i-1,r}$ and $\tilde{\mathcal{U}}_{i,r}$ induced in Corollaries~\ref{coro:inducelambda}~(6) and~\ref{coro:inducegamma}~(6) by our maps $\Lambda$ and $\Gamma$, we obtain for $1\leq i\leq r$:
$$
\sum_{n\geq 0} \tilde{U}_{i,r}(n)q^n=\sum_{(\mu,\lambda) \in \tilde{\mathcal{S}}_{i-1,r}} q^{|(\mu,\lambda)|},$$
and this gives~\eqref{eq:fijW} by~\eqref{Bri2ter}. We derive~\eqref{eq:fij} from~\eqref{eq:fijW} and~\eqref{Bri2bis}. The first part of Corollary~\ref{coro:newbressoud} is then immediate by extracting coefficients in the following identity, obtained from~\eqref{eq:fij} and~\eqref{eq:fijW}:
$$(1+q)\sum_{n\geq 0} \tilde{U}_{i,r}(n)q^n=\frac{1}{(q)_\infty}\left((q^{2r},q^{i+1},q^{2r-i-1};q^{2r})_\infty+q(q^{2r},q^{i-1},q^{2r-i+1};q^{2r})_\infty\right).
$$

\subsection{Proof of Corollary~\ref{coro:list}}

Formulas~\eqref{Br3.3form}--\eqref{fijform3} are derived by using the generating functions in~\eqref{Br3.3bis}--\eqref{BPbisbis} and~\eqref{Br3.3ter}--\eqref{fij4}, together with the bijections between $\mathcal{P}_{i,r}$ and $\mathcal{A}_{i,r}$, $\mathcal{R}_{i,r}$ and $\mathcal{B}_{i,r}$, and $\tilde{\mathcal{R}}_{i,r}$ and $\tilde{\mathcal{B}}_{i,r}$, induced in Corollary~\ref{coro:inducelambda}~(1),~(3),~(4) and Corollary~\ref{coro:inducegamma}~(1),~(3),~(4) by the maps $\Lambda$ and $\Gamma$.



\begin{thebibliography}{99}

\bibitem{Pooneh} P. Afsharijoo, \textit{Looking for a new version of Gordon’s identities}, Ann. Comb. \textbf{25} (2021), 543--57.

\bibitem{ADJM} P. Afsharijoo, J. Dousse, F. Jouhet, and H. Mourtada, {\em New companions to the Andrews–Gordon identities motivated by commutative algebra}, Adv. Math. {\bf 417} (2023), Paper No. 108946, 28 pp.

\bibitem{A74} G. E. Andrews, {\em An analytic generalization of the Rogers--Ramanujan identities for odd moduli}, Proc. Nat. Acad. Sci. USA {\bf 71} (1974), 4082--4085. 

\bibitem{Knot} C. Armond and O. T. Dasbach, \textit{Rogers–Ramanujan identities and the head and tail of the colored Jones polynomial}, preprint (2011), arXiv:1106.3948, 27 pp.

\bibitem{Baxter} R. J. Baxter, \textit{Rogers-Ramanujan identities in the hard hexagon model}, J. Stat. Phys. \textbf{26} (1981), 427--452.

\bibitem{B} D. Bressoud, {\em A generalization of the Rogers--Ramanujan identities for all moduli}, J. Comb. Th. A {\bf 27} (1979), 64--68.

\bibitem{B80} D. Bressoud, {\em An analytic generalization of the Rogers--Ramanujan identities with interpretation}, Q. J. Math {\bf 31} (1980), Issue 4, 385--399.

\bibitem{Br80} D. Bressoud, {\em Analytic and combinatorial generalization of the Rogers--Ramanujan identities}, Mem. Amer. Math. Soc. {\bf 24} (1980), No. 227, 54 pp.

\bibitem{Mourtada} C. Bruschek, H. Mourtada and J. Schepers, {\em Arc spaces and the Rogers–Ramanujan identities}, Ramanujan J. {\bf 30} (2013), 9--38.

\bibitem{DHK} J. Dousse, L. Hardiman and I. Konan, \textit{Partition identities from higher level crystals of $A_1^{(1)}$}, preprint (2021), arXiv:2111.10279, 15 pp., Proc. Amer. Math. Soc., to appear.

\bibitem{DJK} J. Dousse, F. Jouhet, and I. Konan, {\em Bilateral Bailey lattices and Andrews--Gordon type identities}, preprint (2023), arXiv:2307.02346, 27 pp.

\bibitem{GR} G. Gasper and M. Rahman,
{\em Basic Hypergeometric Series}, Second Edition, Encyclopedia of Mathematics
And Its Applications 96, Cambridge University Press, Cambridge, 2004.

\bibitem{Go} B. Gordon, \emph{A Combinatorial Generalization of the Rogers-Ramanujan Identities}, Amer. J. Math. {\bf 83} (1961), 393--399.

\bibitem{HJZ1} T. Y. He, K. Q. Ji, and A. X. H. Zhao,
\emph{Overpartitions and Bressoud’s conjecture, I}, Adv. Math. {\bf 404} (2022), Paper No. 108449, 81 pp.

\bibitem{HJZ2} T. Y. He, K. Q. Ji, and A. X. H. Zhao,
\emph{Overpartitions and Bressoud’s conjecture, II}, preprint (2022),  arXiv:2001.00162, 32 pp.

\bibitem{KLRS} S. Kanade, J. Lepowsky, M. C. Russell, and A. V. Sills, \textit{Ghost series and a motivated proof of the Andrews-Bressoud identities}, J. Combin. Theory Ser. A \textbf{146} (2017), 33--62.

\bibitem{Ki} S. Kim, \emph{Bressoud's conjecture}, Adv. Math. {\bf 325} (2018), 770--813.

\bibitem{Kur}  K. Kur\c{s}ung\"oz, \textit{Bressoud style identities for regular partitions and overpartitions}, J. Number Theory \textbf{168} (2016) 45--63.

\bibitem{LM} J. Lepowsky and S. Milne, \textit{Lie algebraic approaches to classical partition identities}, Adv. Math. \textbf{29} (1978), 15--59.

\bibitem{Lepowsky} J. Lepowsky and R. L. Wilson, \textit{The structure of standard modules, I: Universal algebras and the Rogers-Ramanujan identities}, Invent. Math. \textbf{77} (1984), 199--290.

\bibitem{Lepowsky2} J. Lepowsky and R. L. Wilson, \textit{The structure of standard modules, II: The case $A^{(1)}_1$, principal gradation}, Invent. Math. \textbf{79} (1985), 417--442.

\bibitem{MacMahon} P. A. MacMahon, {\em Combinatory Analysis vol. 2}, Cambridge University Press, New York, NY, USA, 1916.

\bibitem{MP1} A. Meurman and M. Primc, \textit{Annihilating ideals of standard modules of $sl(2,\mathbb{C})^{\sim}$ and combinatorial identities}, Adv. Math., \textbf{64} (1987), 177--240.

\bibitem{MP2} A. Meurman and M. Primc, \textit{Annihilating fields of standard modules of $sl(2,\mathbb{C})^{\sim}$ and combinatorial identities}, Mem. Amer. Math. Soc., \textbf{137} (1999), 89 pp.

\bibitem{RR19} L. J. Rogers and S. Ramanujan, {\em Proof of certain identities in combinatory analysis}, Cambr. Phil. Soc. Proc. {\bf 19} (1919), 211--216.

\bibitem{SchurRR} I. Schur, {\em Ein Beitrag zur Additiven Zahlentheorie und zur Theorie der Kettenbr\"uche}, S.-B. Preuss. Akad. Wiss. Phys. Math. Klasse, 1917, 302--321.

\bibitem{Sills} A. V. Sills, \textit{An Invitation to the Rogers–Ramanujan Identities}, CRC Press, Boca Raton, Florida, 2018.

\bibitem{W} S. O. Warnaar,
\emph{The Andrews--Gordon identities and $q$-multinomial coefficients}, Commun. Math Phys. {\bf 184} (1997), 203--232.

\bibitem{W2} S.O. Warnaar, \textit{The $A_2$ Andrews--Gordon identities and cylindric partitions}, Trans. Amer. Math. Soc., Ser. B \textbf{10} (2023), 715--765. 

\end{thebibliography}
\end{document}